\documentclass[12pt,a4paper,leqno]{amsart}
\usepackage{latexsym,amssymb,amsthm,amscd}
\usepackage{eepic}
\usepackage{color}
\newcommand{\coRR}[1]{#1}
\newcommand{\coGR}[1]{#1}
\definecolor{orange}{rgb}{1,0.5,0}

\allowdisplaybreaks[4]
\newcommand\lsection{\@startsection {section}{1}{\z@}%
                                   {-3.5ex \@plus -1ex \@minus -.2ex}%
                                   {1.0ex \@plus.2ex}%
                                   {\normalfont\large\bfseries}}
\catcode`\@=12
\makeatletter
\newcommand*\verytiny{\@setfontsize\HUGE{2}{3}}
\makeatother
\makeatletter
\@addtoreset{equation}{section}

\makeatother
\newtheorem{thm}{Theorem}[section]
\newtheorem{cor}[thm]{Corollary}
\newtheorem{lem}[thm]{Lemma}
\newtheorem{prop}[thm]{Proposition}
\theoremstyle{definition}
\newtheorem{defn}[thm]{Definition}
\newtheorem{rem}[thm]{Remark}
\newtheorem{exam}[thm]{Example}

\newcommand{\NN}{\mathbb{N}}
\newcommand{\ZZ}{\mathbb{Z}}
\newcommand{\RR}{\mathbb{R}}

\newcommand{\QQ}{\mathbb{Q}}

\newcommand{\LL}{\mathbb{L}}

\newcommand{\comp}{\text{\scriptsize$\circ$}}
\newcommand{\pd}[2]{\frac{\partial{#1}}{\partial{#2}}}

\newcommand{\ord}{\mathop{\mathrm{ord}}\nolimits}
\newcommand{\lcm}{\mathop{\mathrm{lcm}}\nolimits}

\newcommand{\Var}{\mathop{\mathrm{Var}}\nolimits}

\newcommand{\medsum}{\mathop{\textrm{\scriptsize $\mathop\sum$}}}

\renewcommand{\pd}[2]{\frac{\partial#1}{\partial#2}}

\title[\tiny{Motivic invariant\coGR{s} of real polynomial functions and 
Newton polyhedron\coGR{s}}]
{Motivic invariant\coGR{s} of real polynomial functions 
and \coGR{their} Newton polyhedron\coGR{s}}
\author{Goulwen Fichou and Toshizumi Fukui}
\date\today
\thanks{The authors \coGR{have} been supported by Saitama University and the ANR project ANR-08-JCJC-0118-01.}
\address{IRMAR (UMR 6625), Universit\'e de Rennes 1, Campus de 
Beaulieu, 35042 Rennes Cedex, France} 
\address{Department of Mathematics, Faculty of Science, Saitama University,
Saitama, 338-8570, Japan}


\begin{document}
\begin{abstract} 
We give an expression of the motivic zeta function for a real polynomial 
function in terms of the  Newton polyhedron of the function.
As a consequence, we show that the weights are determined by the motivic zeta
function for convenient weighted
homogeneous polynomials in three variables. 
We apply this result to \coGR{the} blow-Nash equivalence.
\end{abstract}
\maketitle

In Singularity Theory, one aim is to classify singular objects with
respect to a given equivalence relation. We focus on the singularities of
function germs, and more precisely on the case of weighted
homogeneous polynomial functions\coGR{;} that is\coGR{,} 
polynomial functions that satisf\coGR{y}
$$
f(t^{w_1}x_1,\ldots,t^{w_n}x_n)=t^df(x_1,\dots,x_n)\quad \textrm { for all real number\coGR{s} }t.
$$
We refer to $(w_1,\dots,w_n)\in \mathbb N^n$ as \coGR{weights} and
to  $d\in \mathbb N$ as the weighted degree of $f$
with respect to $(w_1,\dots,w_n)$. 
\coGR{W}e tackle the question of the invariance of the weights under
a given equivalence relation for germs \coGR{of such functions at the origin}.

Concerning complex analytic function germs, the first result in 
this direction is due to K.~Saito (\cite[Lemma 4.\coRR{3}]{Saito}). 
He proved that\coRR{, among the weighted homogeneous polynomials with 
isolated singularities,} 
the weights are local analytic invariants of the pair 
$(\mathbb C^n,f^{-1}(0))$ at the origin, for 
a weighted homogeneous polynomial $f$. 
Concerning the topological equivalence, E.~Yoshinaga and M.~Suzuki \cite{YoSu} in
1979 (and later T. Nishimura \cite{Ni} in 1986) proved the topological
invariance of the weights in the two \coGR{variable} 
case \coRR{with isolated singularities}, whereas O.~Saeki
\cite{Saeki} in 1988 \coRR{proved that the weights of a weighted homogeneous polynomial $f$ in $\mathbb C^3$ with an isolated singularity are local topological invariants of the pair $(\mathbb C^3,f^{-1}(0))$ at the origin}.

In this paper, we are concerned with the real counterpart of this
question, considering equivalence relations on real analytic function
germs. 
Since the topological equivalence is too weak in the real setting, the most relevant
equivalence relation to consider is the blow-analytic equivalence 
introduced by T.-C. Kuo (cf. \cite{Kuo}, 
and also \cite{FKK,FP} for surveys). Real analytic function germs
$f,g:(\mathbb R^n,0)\longrightarrow (\mathbb R,0)$ are said to be
blow-analytically equivalent in the sense of \cite{Kuo} if
there exist real modifications $\beta_f:M_f \longrightarrow \mathbb R^n$
and $\beta_g:M_g \longrightarrow \mathbb R^n$ and an analytic
isomorphism $\Phi :(M_f, \beta_f^{-1}(0)) \longrightarrow (M_g,
\beta_g^{-1}(0))$ which induces a homeomorphism $\phi:(\mathbb R^n,0)
\longrightarrow  (\mathbb R^n,0)$ such that $f=g \circ \phi$. 

For polynomial functions, or more generally Nash functions
(i.\,e.\,, real analytic functions with semi-algebraic graph\coGR{s}), a natural
counterpart exists, called blow-Nash equivalence, that takes into
account the algebraic nature of Nash functions. This equivalence
relation ha\coGR{s} been proved to have nice properties
(cf. \cite{FKK,Fichou,simple}).

The question of the invariance of the weights for weighted homogeneous
polynomial functions under blow-analytic equivalence already appeared
as a conjecture in \cite{Fukui} and as a question in \cite{KP}. A
positive answer has been given by O. M.~Abderrahmane \cite{Ould} in the two
\coGR{variable} case, using two invariants of blow-analytic equivalence:
Fukui invariants \cite{Fukui} and zeta functions \cite{KP}
constructed by S. Koike and A. Parusi\'nski 
using motivic integration \cite{DL} with the Euler characteristic of the homology of locally finite chains with closed supports as a measure. 

In the case of blow-Nash equivalence, the forthcoming Theorem \ref{w3} states that
convenient three \coGR{variable} weighted homogeneous polynomial functions
that are blow-Nash equivalent must have the same weights. To
prove this, we investigate the zeta function introduced in
\cite{Fichou} as an invariant of the blow-Nash equivalence, using as a measure the virtual Poincar\'e polynomial
\cite{MCP}. This polynomial is an additive and multiplicative
invariant for real algebraic sets, wh\coGR{ose} degree is equal to the
dimension of the variety. 

The main point \coGR{in} the proof of Theorem
\ref{w3} is to estimate the degrees of the coefficient\coGR{s} of the zeta
functions, which \coGR{are} defined using the virtual Poincar\'e polynomial, 
in terms of the Newton polyhedron of a given polynomial
function.
Zeta functions in motivic integration have 
already been computed in
terms of \coGR{the} Newton polyhedron \cite{DK,DL-modif,Guibert}, and our Theorem
\ref{Thm1} is a version that focus\coGR{ses} on \coGR{bounds on the} degrees. 
The main result in this paper, Theorem \ref{LinearBound}, gives a bound for the degree of the
coefficient\coGR{s} of the zeta function, which leads to the notion of
\textit{leading exponent} in section \ref{Le}. In the case of
convenient weighted homogeneous polynomial functions, this leading
exponent gives one more information on the
weights. 
Th\coGR{is  information} will be sufficient to \coGR{conclude} Theorem \ref{w3}.\\
{\bf \coGR{Notation}:}
Throughout the paper, 
$\RR$ denotes the set of real numbers, $\NN$ denotes the set of natural numbers with $0$ 
and we set 
$$
\RR_+=\{x\in\RR:x\ge0\},\qquad
\RR_+^*=\{x\in\RR:x>0\},\qquad
\NN^*=\NN\setminus\{0\}.
$$
\vskip 5mm

{\bf Acknowledgements:} 
The authors would like to thank the anonymous referee for 
\coGR{their useful} comments in order to improve the paper. 

\section{Motivic measure for arc space}
In this section we recall briefly how we can measure \coGR{certain} 
subsets of arc spaces in the context of real geometry, using
the theory of motivic integration as developed by J.~Denef
and F.~Loeser \cite{DL}. 
The real theory has already been developed in \cite{KP,Fichou}.

The measure takes its value in the Grothendieck ring $K_0(\Var_\RR)$ of real algebraic
varieties \cite{MCP}. It is defined as the free abelian group $K_0(\Var_\RR)$ generated by isomorphism
classes $[X]$ of real algebraic varieties modulo the subgroup generated by the
relation $[X]=[Y]+[X\setminus Y]$ for $Y\subset X$ a closed
subvariety of $X$. 
The multiplicative structure comes from the Cartesian product of varieties.

\subsection{Motivic zeta functions}\label{arc-space}

Let $M$ be a nonsingular real algebraic variety and $S$ a subset of $M$.
Consider the space of formal
 arcs with origin in $S$
$$
\mathcal L(M,S):=\{
\alpha:(\RR,0)\to (M,S):\
\textrm{formal}
\}.
$$
We \coGR{write} $\mathcal L(M,x)$ when $S=\{x\}$ is reduced to a point, and
for $m\in \NN$ denote by $L_m(M,x)$ the set of $m$-jets of elements of $\mathcal
L(M,x)$.

Let $p_m:\mathcal{L}(\RR^n,0)\to L_m(\RR^n,0)$ denote the map defined by 
taking $m$-jets. 
For a subset $\mathcal A$ of $\mathcal L(\RR^n,0)$ \coGR{which  is} 
constructible in the sense of \cite{DL2}, 
then, for each $m \in \mathbb N$, $p_m(\mathcal A)$ is Zariski constructible 
and therefore admits a measure in $K_0(\Var_\RR)$. Moreover the limit
$$
\lim_{m\to\infty}\frac{[p_m(\mathcal A)]}{\LL^{mn}}
$$
ha\coGR{s} a \coGR{meaning} in the
localised Grothendieck ring $K_0(\Var_\RR)[[\LL^{-1}]]$, 
where $[p_m(\mathcal A)]$ is the measure of $p_m(\mathcal A)$ in $K_0(\Var_\RR)$
and $\LL$ the measure of the real affine line. We define the measure of $\mathcal A$ to be this limit.

A typical example of such a constructible
subset of $\mathcal L(\RR^n,0)$ is given by the preimage under a truncation map $p_m$ of a Zariski constructible
subset of $L_m(M,x)$ (cf. \cite{DL}).
The subsets of the arc space we
consider in this paper are of this type.

The particular case of spaces of arcs defined by arcs with a specified order
will play a crucial role in this paper.
\begin{exam}\label{ex1}
For $a=(a_1\dots,a_n)\in\NN^n$, we consider the set
$\mathcal{L}_a\subset \mathcal L(\RR^n,0)$ of arcs $\alpha=(\alpha_1,\ldots,\alpha_n)$ in $\mathbb R^n$ whose $i$-th component $\alpha_i$
vanishes if $a_i=0$ or $\alpha_i(t)=ct^{a_i}+\ldots$, where $c\neq 0$, is of order $a_i$ otherwise. Namely
$$
\mathcal{L}_a=\{\alpha\in\mathcal L(\RR^n,0):
\ord \alpha_i=a_i\textrm{ if } i\in I(a),\ 
\alpha_i= 0\textrm{ if } i\not\in I(a)\}
$$
where $I(a)=\{i:a_i>0\}$. 
If $m$ is greater than the maximal value of $a_i$, for $i=1,\ldots,n$,
then $p_m(\mathcal L_a)$ is isomorphic to
$\RR^{m|I(a)|-\medsum_{i}a_i}\times (\RR^*)^{|I(a)|}$ where $|I(a)|$
denote{\coGR{s}} the cardinal of $I(a)$. Then
$$
[p_m(\mathcal L_a)]=
\LL^{m|I(a)|-\medsum_{i}a_i}(\LL-1)^{|I(a)|}
=(\LL-1)^{|I(a)|}\LL^{m|I(a)|-s(a)}
$$ 
where 
$s(a)=\medsum_{i=1}^na_i$, and therefore 
\begin{align*}
[\mathcal L_a]=
\lim_{m\to\infty}
\frac{[p_m(\mathcal L_a)]}{\LL^{mn}}=&
\begin{cases}
(\LL-1)^n\LL^{-s(a)}&\textrm{if }\ |I(a)|=n,\\
0&\textrm{if }\ |I(a)|<n.
\end{cases}
\end{align*}
\end{exam}

In other words, truncated arcs with some components equal to zero, can be seen as the image under truncation of arcs with higher orders.
We will use this remark in order to compute in
section \ref{poly} the arc spaces associated with a given real polynomial
function germ.

Let $f:(\RR^n,0)\to (\RR,0)$ be a polynomial
function germ, \coGR{and f}or $k\in \mathbb N^*$, we define the arc space $\mathcal
A_k(f)\subset \mathcal L(\RR^n,0)$ by
$$
\mathcal A_k(f)=
\{\alpha\in\mathcal L(\RR^n,0):f\comp\alpha(t)=ct^k+\cdots,\ c\ne0\}.$$
Similarly, we \coGR{define} arc spaces with sign 
$\mathcal A^+_k(f)\subset \mathcal A_k(f)$ and 
$\mathcal A^-_k(f)\subset \mathcal A_k(f)$\coGR{,} 
by
$$
\mathcal A^+_k(f)=\{\alpha\in\mathcal L(\RR^n,0):f\comp\alpha(t)=ct^k+\cdots,\ c=1\}
$$
and
$$
\mathcal A^-_k(f)=\{\alpha\in\mathcal L(\RR^n,0):f\comp\alpha(t)=ct^k+\cdots,\ c=-1\}.
$$
As the treatment of the spaces with sign $\mathcal A^+_k(f)$ and $\mathcal A^-_k(f)$ is similar, we denote in the following by $\mathcal A^\pm_k(f)$ one of these two spaces.
Since the $k$-jet of $\alpha\in\mathcal L(\RR^n,0)$ 
determines the $k$-jet of $f\comp\alpha$, we obtain an expression for
the measure of $\mathcal A_k(f)$ and of $\mathcal A^\pm_k(f)$ in terms of 
\coRR{the space of $m$-jets of arcs with $m\ge k$}. 
More precisely
$$
[\mathcal A_k(f)]=\frac{[p_m(\mathcal A_k(f))]}{\LL^{mn}}\ 
\textrm{ and }\ 
[\mathcal A^\pm_k(f)]=\frac{[p_m(\mathcal A^\pm_k(f))]}{\LL^{mn}}
$$
for $m\ge k$. 
The associated zeta function and zeta functions with sign are the formal power series with coefficients in $K_0(\Var_\RR)[[\LL^{-1}]]$ defined by 
$$
Z(f)=\medsum_{k=1}^\infty[\mathcal A_k(f)]t^k\ \textrm{ and }\ 
Z^\pm(f)=\medsum_{k=1}^\infty[\mathcal A^\pm_k(f)]t^k. 
$$

\begin{exam} 
Consider the one variable polynomial function defined by 
$f(x)=x^d$ with $d \in \NN^*$. Then we have 
$$
[\mathcal A_k(f)]=
\begin{cases}
(\LL-1)\LL^{-l}& \textrm{if }\  k=ld,~~l\in \NN, \\
0& \textrm{if }\ d\nmid k.
\end{cases}
$$
As a consequence, the zeta function of $f$ is equal to 
$$Z(f)=\LL^{-1}\frac{(\LL-1)t^d}{1-\LL^{-1}t^d}.$$
\end{exam}

\subsection{\coGR{The v}irtual Poincar\'e polynomial}\label{virt}

For real algebraic varieties, the best additive invariant known
is the virtual Poincar\'e polynomial \cite{MCP}. 
It assigns a polynomial $\beta(X)$ with integer coefficients 
to a (real algebraic) Zariski constructible set $X$
in such a way that\coGR{:} 
\begin{itemize}
\item $\beta(X)=\beta(Y)+\beta(X\setminus Y)$ for $Y \subset
  X$ a closed subvariety of $X$ (additivity), 
\item $\beta(X\times Y)=\beta(X)\beta(Y)$ (multiplicativity), and 
\item the coefficients coincide 
with the Betti numbers with $\ZZ_2$-coefficients for 
compact nonsingular real algebraic sets.
\end{itemize}
\begin{prop}[\cite{MCP}]\label{beta} Take $i\in \mathbb N$. The Betti
  number $\beta_i(\cdot)=\dim H_i(\cdot, {\mathbb Z}_2)$, 
considered on compact nonsingular real algebraic sets,
  admits \coGR{a} unique extension as an additive map $\beta_i$ to the
  category of Zariski constructible sets, 
with values in $\mathbb Z$. Namely
$$\beta_i(X)=\beta_i(Y)+\beta_i(X \setminus Y)$$
for $Y \subset X$ a closed subvariety of $X$. 

Moreover the polynomial
$\beta(\cdot)=\sum_{i \geq 0} \beta_i(\cdot)u^i \in \mathbb Z [u]$ is
multiplicative 
$$\beta(X\times Y)=\beta(X)\beta(Y)$$ 
for $X,Y$ Zariski constructible sets.
\end{prop}

The invariant $\beta_i$ is called the $i$-th virtual Betti number, and the
polynomial $\beta$ the virtual Poincar\'e polynomial. By evaluati\coGR{ng} of
the virtual Poincar\'e polynomial at $u=-1$ one recovers the Euler
characteristic for Borel-Moore homology, that is, the homology of
locally finite chains with closed supports \cite{MCP}.
The following simple example illustrates the way to compute, 
in \coGR{practice}, the virtual Poincar\'e polynomial. 

\begin{exam} 
Let $\mathbb P^n$ denote the real projective space of dimension $n$, 
which is nonsingular and compact. Then 
$\beta(\mathbb P^n)=1+u+\cdots+u^n$, since 
$$
\dim H_i(\mathbb P^n, {\mathbb Z}_2)
=
\begin{cases} 
1&\textrm{for $i \in \{0,\ldots,n\}$, and}\\
0&\textrm{otherwise.}
\end{cases}
$$ 
Now, compactify the affine line $\mathbb  A_{\mathbb R}^1$ in $\mathbb P^1$
 by adding one point $P$ \coGR{at infinity}. 
By additivity $\beta(\mathbb  A_{\mathbb R}^1)=\beta(\mathbb P^1)-\beta(P)=u$,  
and so $\beta(\mathbb A_{\mathbb R}^n)=u^n$ by multiplicativity.
\end{exam}

A crucial property of the virtual Poincar\'e polynomial is that the
degree of $\beta(X)$ is equal to the dimension of 
the Zariski constructible 
set $X$. 
In particular, the virtual Poincar\'e polynomial of a constructible set $X$ cannot be equal to zero as soon as the set of real points of $X$ is not empty. 
We consider in this paper the virtual Poincar\'e polynomial of the (infinite dimensional) spaces of arcs $\mathcal A_k(f)$ and $\mathcal A^{\pm}_k(f)$ associated with a polynomial function $f$, for $k\in \mathbb N^*$. It is defined, e.g. in the case without sign and as explained in section \ref{arc-space}, by the formula
$$\beta(\mathcal A_k(f))=\lim_{m\to\infty}\frac{\beta(p_m(\mathcal A_k))}{u^{mn}}.$$
We
will mainly focus on the degree of $\beta(\mathcal A_k(f))$ in section
\ref{estim-deg}, and we will also \coGR{discuss} its zeros in
section \ref{reco}.

\section{Arc spaces and Newton polyhedron\coGR{s}}
In this section we are interested in expressing the measure of the arc
spaces associated with a polynomial function germ in terms of its Newton polyhedron. Similar
results have already been obtained in \cite{DK,DL-modif,Guibert}. Here we focus mainly on a
formula that will enable us to estimate efficiently the degree of the
virtual Poincar\'e polynomial of the arc spaces in terms of the Newton
polyhedron of the germ.

We start with introducing some standard notation for the Newton
polyhedron associated with a polynomial.

\subsection{\coGR{The} Newton polyhedron}\label{Np}
Let $f:\RR^n\to\RR$ denote a polynomial function.
Consider its Taylor expansion at the origin of $\RR^n$
$$
f(x)=\medsum_{\nu \in \NN^n}c_\nu x^\nu,
$$
where $x^\nu=\prod_{i=1}^n {x_i}^{\nu_i}$ with
$\nu=(\nu_1,\dots,\nu_n) \in \NN^n$, and $c_\nu \in \RR$.
Let $\Gamma_+(f)$ denote the \textbf{Newton polyhedron} of $f$,
defined as the convex hull of the set
$$
\cup_{\nu \in \NN^n}(\nu+\RR_+^n):c_\nu\ne0, 
$$
where $\RR_+$ stands for $[0,+\infty)$.
The \textbf{Newton boundary} $\Gamma(f)$ of $f$ is the union of the compact faces of $\Gamma_+(f)$. 
We \coGR{write} $\gamma < \Gamma(f)$
\coGR{to denote that} 
the compact face $\gamma$ \coGR{belongs} to $\Gamma(f)$, 
\coGR{and} $\gamma < \sigma$ \coGR{for} the inclusion of two faces.
For $a=(a_1,\dots,a_n)\in \mathbb R^n_+$ and 
$\nu=(\nu_1,\dots,\nu_n)\in \mathbb R^n$, we set 
$\langle a,\nu\rangle=\sum_{i=1}^n a_i\nu_i$ and define the
multiplicity $m_f(a)$ of \coRR{$f$ relative to $a$} by
$$
m_f(a)=\min\{\langle a,\nu\rangle :\nu\in\Gamma_+(f)\}.
$$
For $a\in \mathbb R_+^n$, we define the face $\gamma_f(a)$ 
of the Newton polyhedron of $f$ associated with $a$ by
$$
\gamma_f(a)=\{\nu\in\Gamma_+(f):\langle a,\nu\rangle=m_f(a)\},$$
and for a subset $S$ of $\NN^n$ we set
$$
f_S(x)=\medsum_{\nu\in S}c_\nu x^\nu.
$$
We define an equivalence relation \coGR{o}n $\RR_+^n$ by 
$$
a\sim b \quad\Longleftrightarrow\quad \gamma_f(a)=\gamma_f(b)\coGR{.}
$$
The
partition of $\RR_+^n$ given by the equivalence classes of this
equivalence relation is the \textbf{dual Newton diagram} of $f$ and it
is denoted by $\Gamma^*(f)$. This defines a cone
subdivision of $\RR_+^n$, which we identify 
with the dual Newton polyhedron throughout the paper. Note that the function $m_f$ is a piecewise linear function on $\mathbb R^n_+$, which is linear on any cone belonging to this subdivision. In particular, on the cone corresponding to a face $\gamma$ of $\Gamma_+(f)$, the function $m_f$ is given by the scalar product $\langle \cdot,\nu \rangle$ with any $\nu \in \gamma$. 
Let $\Lambda(f)\subset \NN^n$ denote the set of primitive generators of the
1-cones of $\Gamma^*(f)$. 

The polynomial $f$ is said to be
\textbf{convenient} if the monomials $x_i^{p_i}$, for $i=1,\ldots,n$
and some $p_i\in \mathbb
N^*$, appear in
the expression of $f$ with non-zero coefficients. 

We say that the polynomial function $f$ is \textbf{non-degenerate} 
if, \coGR{for any compact face $\gamma$ of $\Gamma_+(f)$,} 
all singular points of $f_\gamma$ are contained in the union of 
some coordinate hyperplanes. 
Namely $f$ is \textbf{non-degenerate} if
$$
\biggl(
\pd{f_\gamma}{x_1}(c),\dots,\pd{f_\gamma}{x_n}(c)
\biggr)\ne(0,\dots,0)
$$
for all $c\in (\RR^*)^n$ with $f_\gamma(c)=0$, where $\gamma$ 
\coGR{is any} compact \coGR{face} of $\Gamma_+(f)$.

\coGR{Finally, for any} compact face \coGR{$\gamma$} of 
the Newton polyhedron of $f$\coGR{, w}e define algebraic subsets $X_\gamma$, $X^+_\gamma$ and $X^-_\gamma$ of $(\RR^*)^n$ 
\coGR{by}
$$
X_\gamma=\{c\in(\RR^*)^n:f_\gamma(c)=0\},
\quad\textrm{ and }\quad
X^{\pm}_\gamma=\{c\in(\RR^*)^n:f_\gamma(c)=\pm1\}.
$$
\begin{rem}\label{set-intro} 
There exists an algebraic variety 
$\widehat{X}_\gamma$
in $(\RR^*)^{\dim\gamma}$ so that
$$
X_\gamma\simeq (\RR^*)^{n-\dim\gamma}\times\widehat{X}_\gamma
$$
(cf. \cite{Arnold}, Vol. 2, Part II, Chapter 8 for example). Taking measure\coGR{s}, we therefore have
$$
[X_\gamma]=(\LL-1)^{n-\dim\gamma}[\widehat{X}_\gamma]. 
$$\end{rem}
\begin{exam}
Let $X$ denote the zero set of $f(x_1,x_2)=x_1^2+x_2^3$ in $(\RR^*)^2$. 
We have $X\simeq\RR^*\times \widehat{X}$
where $\widehat{X}=\{y_2\in\RR^*:y_2+1=0\}$.
This follows from the \coGR{identity} 
$f\comp\varphi(y_1,y_2)=y_1^6y_2^2(1+y_2)$, 
where $\varphi:(\RR^*)^2\to(\RR^*)^2$ is \coGR{the} algebraic isomorphism 
defined by $(y_1,y_2)\mapsto(x_1,x_2)=(y_1^3y_2,y_1^2y_2)$. 
\end{exam}
\begin{rem}
If $f$ is non-degenerate, then $X_\gamma$ (resp. $\widehat{X}_\gamma$) is a non-singular submanifold of $(\RR^*)^n$ (resp. $(\RR^*)^{\dim\gamma}$) of codimension 1, whenever it is not empty. Since $f_\gamma$ is weighted homogeneous, 
$$
\Bigl\{\pd{f_\gamma}{x_1}=\cdots=\pd{f_\gamma}{x_n}=0\Bigr\}=
\Bigl\{\pd{f_\gamma}{x_1}=\cdots=\pd{f_\gamma}{x_n}=f_\gamma=0\Bigr\}
$$
and the non-degeneracy of $f$ also implies that the varieties $X^\pm_\gamma$ are nonsingular manifolds of codimension 1, whenever they are not empty. Moreover, the differential of $f_\gamma$ is nonzero. Thus we have\coGR{:} 
\begin{itemize}
\item $X^+_\gamma$ and $X^-_\gamma$ cannot be empty simultaneously\coGR{, and}
\item \coGR{i}f $X_\gamma$ is not empty, then both $X^+_\gamma$ and $X^-_\gamma$ are not empty.
\end{itemize}
\end{rem}
\subsection{Motivic invariant of a polynomial function}\label{poly}
We express the measure of the arc spaces associated with a
polynomial function $f$ in terms of its Newton polyhedron. 
The set of integers $k\in \mathbb N^*$ for which the arc space $\mathcal A_k(f)$ 
is not empty has already been studied in the context of blow-analytic
equivalence and is called the set of Fukui invariants \cite{Ould,KP}. 
It coincides with the set of exponents that 
\coGR{appear} in the zeta function of $f$ with non-zero coefficients. We denote by
$$
A(f)=\{k\in \NN^* :\mathcal A_k(f)\ne\emptyset\}$$
the set $A(f)$ of Fukui invariants of $f$ and by
$$
A^\pm(f)=\{k\in \NN^* :\mathcal A_k^\pm(f)\ne\emptyset\}. 
$$ 
the sets $A^+(f)$ and $A^-(f)$ of Fukui invariants with sign of $f$.

Next lemma (adapted to the real setting from \cite{Guibert}, Lemma 2.2.1) computes the measure of the arc spaces associated with $f$
for arcs with a specified order $a\in \mathbb N^n$. 
As illustrated in Example \ref{ex1}, those arcs with order $a\in (\NN^*)^n$ play the most important role when computing the measure.

\begin{lem}\label{KeyLem}
Assume that $f$ is a non-degenerate polynomial. 
Take $a\in (\mathbb N^*)^n$ and $k\in
\NN^*$. The measure of $\mathcal L_a\cap \mathcal A_k(f)$ is given by
\begin{align*}
[\mathcal L_a\cap \mathcal A_k(f)]=&
\begin{cases}
0&\textrm{if }\ m_f(a)>k,\\
\bigl((\LL-1)^n-[X_{\gamma(a)}]\bigr)\, 
\LL^{-s(a)}&\textrm{if }\ m_f(a)=k,\\
(\LL-1)\,[X_{\gamma(a)}]\,\LL^{-s(a)-k+m_f(a)}&\textrm{if }\ m_f(a)<k.
\end{cases}
\end{align*}
In the case with sign, the measure of $\mathcal L_a\cap \mathcal A^\pm_k(f)$ is given by
\begin{align*}
[\mathcal L_a\cap \mathcal A^\pm_k(f)]=&
\begin{cases}
0&\textrm{if }\ m_f(a)>k,\\
[X^\pm_{\gamma(a)}]\,\LL^{-s(a)}&\textrm{if }\ m_f(a)=k,\\
[X_{\gamma(a)}]\,\LL^{-s(a)-k+m_f(a)}&\textrm{if }\ m_f(a)<k.
\end{cases}
\end{align*}
\end{lem}
\begin{proof} 
Since the proof of the case with sign is similar to the
case without sign, we treat only the \coGR{later}.
For an arc $\alpha\in\mathcal{L}_a$ with order $a=(a_1,\ldots,a_n)$, we compute its
composition with $f$ in terms of data coming from the Newton
polyhedron of $f$. Define $\phi(t)=(\phi_1(t),\dots,\phi_n(t))$ by 
$$
\alpha(t)=(t^{a_1}\phi_1(t),\dots,t^{a_n}\phi_n(t)),
$$
so that $\phi_i(0)\ne0$ for $i=1,\dots,n$. Using the Taylor expansion
for $f$, we obtain 
$$
f(\alpha(t))
=\medsum_{\nu \in \NN^n} c_\nu
\prod_{i=1}^n (t^{a_i}\phi_i(t))^{\nu_i}\\
=\medsum_{\nu \in \NN^n} c_\nu
\phi(t)^{\nu}t^{\langle a,\nu\rangle}.
$$
As the order of $f(\alpha(t))$ is greater than $m_f(a)$, we \coGR{may write} 
$$
f(\alpha(t))=t^{m_f(a)}(f_{\gamma(a)}(\phi(t))+R(t))
$$
where $R(t)$ is a power series, depending on $\phi$, with strictly positive order. In particular, if $k$ is strictly less than $m_f(a)$, then $\mathcal L_a\cap \mathcal A_k(f)$ is empty. If $k$ is equal to $m_f(a)$, then $\mathcal L_a\cap \mathcal A_k(f)$ is described by those arcs
$\alpha$ with $f_{\gamma(a)}(\phi(0))\neq 0$, meaning that $\phi (0)\in (\RR^*)^n \setminus X_{\gamma(a)}$. To compute the measure of $\mathcal L_a\cap \mathcal A_k(f)$, note that the quotient 
$$\frac{[p_m(\mathcal L_a\cap \mathcal A_k(f))]}{\LL^{mn}}$$
stabilizes for $m\geq k$, and for $m=k$ we have just seen that
$$p_{k}(\mathcal L_a\cap \mathcal A_k(f))\simeq \big( (\RR^*)^n \setminus X_{\gamma(a)}\big) \times \RR^{\sum_{i=1}^n(k-a_i)}$$
so that 
$$[\mathcal L_a\cap \mathcal A_k(f)]=\frac{[\big( (\RR^*)^n \setminus X_{\gamma(a)}\big) \times \RR^{\sum_{i=1}^n(k-a_i)}]}{\LL^{{k}n}}=\bigl((\LL-1)^n-[X_{\gamma(a)}]\bigr)\, 
\LL^{-s(a)}.$$

We focus now on the case where $k$ is strictly bigger than $m_f(a)$. 
In this case, \coGR{setting $F(t)=t^{-m_f(a)}f(\alpha(t))$, }
the arc $\alpha$ belongs to $\mathcal L_a\cap \mathcal A_k(f)$ 
if and only if $\phi (0)$ belongs to $X_{\gamma(a)}$\coGR{,}
$$
F'(0)=F''(0)=\cdots=F^{(k-m_f(a)-1)}(0)=0\textrm{, \ and }\ 
F^{(k-m_f(a))}(0)\ne0\coGR{.}
$$ 
Indeed, the\coGR{y} implies that the order of $f\circ \alpha$ is 
\coGR{equal} to $k$. 
That system determines the coefficients of $\phi$ (and so of $\alpha$), 
and in particular enables us to compute the measure of 
$\mathcal L_a\cap \mathcal A_k(f)$ in case $k>m_f(a)$. 
Since $F(t)=f_{\gamma(a)}(\phi(t))+R(t)$, we obtain that 
$$F'(t)=\medsum_{i=1}^n\phi_i^{'}(t)\pd{f_{\gamma(a)}}{x_i}(\phi(t))+R'(t).$$
Choosing for $\phi(0)$ any point in $X_{\gamma(a)}$, there exists at least 
one \coGR{partial} derivatives of $f_{\gamma(a)}$, 
say $\pd{f_{\gamma(a)}}{x_{i_0}}$, 
that does not vanish at $\phi(0)$ because $f$ is non-degenerate. 
This shows that we can freely choose the remaining coefficients of $\phi$ 
as soon as we fix the \coGR{value} of $\phi'_{i_0}(0)$ 
so that $F'(0)=0$. 
We proceed in a similar way for higher order derivatives of $F$. 
The computation of $F^{(2)}$ gives
\begin{align*}
F^{(2)}(t)
=&\medsum_{i=1}^n\phi_i^{(2)}(t)\pd{f_{\gamma(a)}}{x_i}(\phi(t))
+\medsum_{i,j=1}^n\phi_i^{'}(t)\phi_j^{'}(t)
\frac{\partial^2 f_{\gamma(a)}}{\partial x_i \partial x_j}(\phi(t))
+R^{(2)}(t)\\
=&\medsum_{i=1}^n\phi_i^{(2)}(t)\pd{f_{\gamma(a)}}{x_i}(\phi(t))+R_2(t)
\end{align*}
where $R_2(t)$ is equal to the sum of the last two terms in the first line. 
Repeating the procedure for any $l \in\{1,\ldots,k-m_f(a)\}$, 
there exists a power series $R_l$ such that
$$
F^{(l)}(t)=\medsum_{i=1}^n\phi_i^{(l)}(t)\pd{f_{\gamma(a)}}{x_i}(\phi(t)) +R_l(t).
$$
In particular, an arbitrary choice of the coefficients of $\phi$ 
enables us to solve the system $F'(0)=F''(0)=\cdots=F^{(k-m_f(a)-1)}(0)=0$ 
as soon as we fix the values of $\phi'_{i_0}(0),\ldots,\phi^{(k-m_f(a)-1)}_{i_0}(0)$. 
Finally, to take into account the additional \coGR{fact} $F^{(k-m_f(a))}(0)\ne0$ 
in order to guarant\coGR{ee} 
that $\alpha$ belongs to $\mathcal L_a\cap \mathcal A_k(f)$, 
we simply need to exclude one value for $\phi^{(k-m_f(a))}_{i_0}(0)$. 
We conclude \coGR{that} 
the measure of the truncation of 
$\mathcal L_a\cap \mathcal A_k(f)$ stabilizes \coGR{sufficiently large} $k$, 
and that
$$
p_{k}(\mathcal L_a\cap \mathcal A_k(f))\simeq 
X_{\gamma(a)}\times \RR^* \times \RR^{kn-s(a)-(k-m_f(a))},
$$
so that 
$$
[\mathcal L_a\cap \mathcal A_k(f)]=(\LL-1)\,[X_{\gamma(a)}]\,\LL^{-s(a)-k+m_f(a)}.
$$
\end{proof}

\begin{rem} The function $f_{\gamma}$ associated with the face $\gamma$ of 
the Newton polyhedron of $f$ is said to be not definite 
if $X_{\gamma}\neq \emptyset$. 
Set 
$$
m_0(f)=\min\{m_f(a):
a\in (\mathbb N^*)^n,~~f_{\gamma(a)} \textrm{ is not definite}
\}.
$$
Then Lemma \ref{KeyLem} (in the case ``$k>m_f(a)$") shows that any integer 
greater than or equal to $m_0(f)$ belongs to the set of Fukui
invariant $A(f)$ of $f$. We put
$T(f)=\{m\in\NN:m\ge m_0(f)\}$, so that $T(f)\subset A(f)$. 
The remaining integers contained in the set of Fukui
invariants coincide with the sets
\begin{align*}
S(f)=&\{m_f(a):a\in (\mathbb N^*)^n,\exists c\in(\RR^*)^n, f_{\gamma(a)}(c)\ne0\},\\
S^\pm(f)=&\{m_f(a):a\in (\mathbb N^*)^n,\exists c\in(\RR^*)^n, f_{\gamma(a)}(c)=\pm1\},
\end{align*}
as illustrated by Lemma \ref{KeyLem} (in the case "$k=m_f(a)$"). As a consequence 
$$
A(f)=S(f)\cup T(f)\ \textrm{ and }\ 
A^\pm(f)=S^\pm(f)\cup T(f).
$$
\end{rem}

Another consequence of Lemma \ref{KeyLem} is a nice description of the measure of the arc
spaces associated with $f$ in terms of the geometry of $f$ and of the
combinatorics of the Newton polyhedron of $f$. 
For a compact face $\gamma$ of
$\Gamma_+(f)$ and $k \in \mathbb N^*$, we define elements
$P_k(\gamma)$ and $Q_k(\gamma)$ in the 
locali\coGR{z}ed Grothendieck ring of real algebraic variety by setting
\begin{align*}
P_k(\gamma)=&
\medsum_{a\in (\mathbb N^*)^n,~\gamma(a)=\gamma,\ m_f(a)=k}\,\LL^{-s(a)}
\end{align*}
and
\begin{align*}
Q_k(\gamma)=&
\medsum_{a\in (\mathbb N^*)^n,~\gamma(a)=\gamma,\ m_f(a)<k}\,\LL^{-k+m_f(a)-s(a)}.
\end{align*}

\begin{thm}\label{Thm1} Let $k\in \NN^*$.
If $f$ is a non-degenerate polynomial, the measure of the arcs spaces
associated with $f$ can be decomposed into \coGR{several terms associated with} 
the compact faces of $\Gamma_+(f)$ as
$$
[\mathcal A_k(f)]
=
\medsum_{\gamma < \Gamma(f)}\,\bigl((\LL-1)^{n}-[X_{\gamma}]\bigr)\,P_k(\gamma)
+(\LL-1)\medsum_{\gamma < \Gamma(f)}\,[X_{\gamma}]\,Q_k(\gamma)
$$
In the case with sign, we obtain similarly
$$
[\mathcal A_k^\pm(f)]
=\medsum_{\gamma < \Gamma(f)}\,[X_\gamma^\pm]\ P_k(\gamma)
+\medsum_{\gamma < \Gamma(f)}\,[X_\gamma]\ Q_k(\gamma).
$$
\end{thm}

\begin{rem}
This decomposition of the measure of the arc spaces
$\mathcal{A}_k(f)$ and $\mathcal A_k^\pm(f)$, where $k\in \mathbb N^*$,
into a sum of two terms is motivated by the difference between\coGR{:} 
\begin{itemize}
\item 
arcs of order $a\in (\mathbb N^*)^n$ that
directly contribute to the coefficient of $t^{k}$, i.e. with $k={m_f(a)}$, and 
\item 
arcs of order $a\in (\mathbb N^*)^n$ that contribute to a term of order $k>m_f(a)$.
\end{itemize} 
\coGR{Understanding} both contributions will be the main step
in section \ref{weight} in order to recover the weights of a weighted
homogeneous polynomial from the zeta functions.
\end{rem}

\begin{rem}\label{levelh} Let us rewrite $P_k(\gamma)$ as
$$
P_k(\gamma)
=\medsum_{a\in (\mathbb N^*)^n,~\gamma(a)
=\gamma,\ m_f(a)=k}\,\LL^{-k+m_f(a)-s(a)},
$$
since $m_f(a)$ is equal to $k$ for those $a$ \coGR{appearing} in the summation. 
Using this expression of $P_k(\gamma)$, 
we see that the difference between $P_k(\gamma)$ and
$Q_k(\gamma)$ lies in the value of $m_f$ \coGR{at} $a\in (\mathbb N^*)^n$.
In particular, in order to understand the powers of $\LL$ in the expression\coGR{s} 
of $P_k(\gamma)$ and $Q_k(\gamma)$, 
we are lead to focus on the levels of
the piecewise linear function $m_f-s$ defined on the dual of the
Newton polyhedron of $f$, and more precisely on the subsets defined by
$m_f=k$ and $m_f<k$ for a given integer $k\in \mathbb N^*$. In the
sequel, we denote by
$h$ the function $h=m_f-s$.
\end{rem}

\begin{proof}[Proof of Theorem \ref{Thm1}]
We decompose the measure of $\mathcal A_k(f)$ with respect to the order
of the arcs and reco\coGR{nstitute} it with respect to the belonging of these
orders to the different compact faces of the Newton Polyhedron of $f$. 
First, it follows from the computation in Example \ref{ex1} that
$$
[\mathcal A_k(f)]=\medsum_{a\in (\mathbb N^*)^n}\,
[\mathcal L_a\cap\mathcal A_k(f)],
$$
since the measure of $[\mathcal L_a\cap\mathcal A_k(f)]$ 
is equal to zero for $a\in \NN^n\setminus (\NN^*)^n$.
Therefore, we obtain from Lemma
\ref{KeyLem} that
\begin{align*}
[\mathcal A_k(f)]
=&
\medsum_{\gamma < \Gamma(f)}\,
\bigl((\LL-1)^{n}-[X_{\gamma}]\bigr)
\medsum_{a\in (\mathbb N^*)^n, ~\gamma(a)=\gamma,\ m_f(a)=k}\,\LL^{-s(a)}\\
&+
(\LL-1)
\medsum_{\gamma < \Gamma(f)}\,[X_{\gamma}]\,
\medsum_{a\in (\mathbb N^*)^n,~\gamma(a)=\gamma,\ m_f(a)<k}
\LL^{-k+m_f(a)-s(a)}. 
\end{align*}
Similarly, 
in the case with sign, we obtain that 
\begin{align*}
[\mathcal A_k^\pm(f)]
=&
\medsum_{a\in (\mathbb N^*)^n}\,
[\mathcal L_a\cap\mathcal A_k^\pm(f)]\\
=&
\medsum_{\gamma < \Gamma(f)}\,
[X_{\gamma}^\pm]
\medsum_{a\in (\mathbb N^*)^n,~\gamma(a)=\gamma,\ m_f(a)=k}\, 
\LL^{-s(a)}\\
&+
\medsum_{\gamma < \Gamma(f)}\,
[X_{\gamma}]\ 
\medsum_{a\in (\mathbb N^*)^n,~\gamma(a)=\gamma,\ m_f(a)<k}
\LL^{-k+m_f(a)-s(a)}. 
\end{align*}
\end{proof}

\section{Estimate of degrees}\label{estim-deg}

This section is the heart of the paper. We give a linear bound for
the degree of the virtual Poincar\'e polynomial of the arc spaces
$\mathcal A_k(f)$ and $\mathcal A_k^{\pm}(f)$
associated with a given polynomial $f$, and we investigate when this bound is
sharp, in the sense that the equality holds for infinitely many $k$.  
Using this bound, we define the \textit{leading exponent} of 
$\beta(\mathcal A_k(f))$\coGR{; t}his leading exponent 
is encoded in the zeta function of $f$. 

Recall that $u$ stands for the Poincar\'e polynomial of the affine
line. In the sequel the notation $P_k(\gamma)$ and $Q_k(\gamma)$
introduced for Theorem \ref{Thm1} will be interpreted with the virtual
Poincar\'e polynomial rather than the measure in the localised
Grothendieck ring (namely $\LL$ is replaced by $u$).

By convention, we define $\dim\emptyset=-\infty$ and
$\max\emptyset=-\infty$.

\subsection{\coGR{Notation}}

We are going to prepare several \coGR{pieces of notation} 
continuing th\coGR{ose introduced} in section \ref{Np}. 
Since $m_f(a)=\min\{\langle a,\nu\rangle:\nu\in \Gamma_+(f)\}$ 
for $a\in \RR_+^n$, the smallest component of 
$a\in \mathbb R_+^n$ is equal to zero if $m_f(a)=0$. 
This means that \coGR{if} $a$ \coGR{belongs} to $(\RR_+^*)^n$\coGR{, then} 
$m_f(a)>0$. We define $e_f$ by 
\begin{equation}\label{Def:Ef}
e_f=\sup\Bigl\{1-\frac{s(a)}{m_f(a)}:a\in(\RR_+^*)^n\Bigr\}.
\end{equation}
Th\coGR{e} number $e_f$ will enable us to give a bound for 
the degree of the virtual Poincar\'e polynomial of the arc spaces. 
\coGR{Finaly, w}e denote by $\Lambda_{\max}(f)\subset \NN^n$ 
the primitive generators of $1$-cones in $\Gamma^*(f)$ that realise $e_f$\coGR{;} 
that is, 
$$
\Lambda_{\max}(f)=\{v\in \Lambda(f): m_f(v)>0,~~e_f=1-\frac{s(v)}{m_f(v)}\}.
$$
\begin{rem}
If $f$ is convenient, then the facets are supported by primitive vectors $v$ 
with $v\in(\NN^*)^n$
or $e_i$ ($i=1,\dots,n$). Since $m_f(e_i)=0$, 
we conclude that $\Lambda_{\max}(f)\subset (\NN^*)^n$. 
If $f$ is not convenient, then there are vectors 
$v\in \Lambda(f)\subset(\NN^n\setminus(\NN^*)^n)$ with $m(v)>0$. 
These $v$ support non-compact facets. Sometimes such $v$ belongs to 
$\Lambda_{\max}(f)$ and
it may happen that $\Lambda_{\max}(f)$ \coGR{only} contains such elements. 
\coGR{Thus,} the set $\Lambda_{\max}(f)\cap(\NN^*)^n$ could be empty. 
\end{rem}
\begin{exam} 
For the two \coGR{variable} function $f(x,y)=x^5+x^2y$, 
the set of primitive generators of $1$-cones in $\Gamma^*(f)$ 
is $\Lambda(f)=\{(1,0),(1,3),(0,1)\}$. 
The number $e_f$ is equal to $1/2$ and $\Lambda_{\max}(f)=\{(1,0)\}$. 
\end{exam}
\begin{exam}
If $f(x,y,z)=xy^d+yz^d+\coRR{zx^d}$, we have 
$$
\Lambda(f)=\{(1,0,0),(0,1,0),(0,0,1),(1,d,0),(0,1,d),(d,0,1),(1,1,1)\}
$$
and $e_f=1-\min\{\frac{d+1}{d},\frac3{d+1}\}=1-\frac3{d+1}$. So we obtain that 
$\Lambda_{\max}(f)=\{(1,1,1)\}$. 
\end{exam}
\begin{exam}
For $f=(x^3+y^3)z+x^d+y^d$ ($d\ge4$), we have 
$$
\Lambda(f)=\{(1,0,0),(0,1,0),(0,0,1),(1,1,0),(1,1,d-3)\}
$$ 
and $e_f=1-\min\{\frac23,\frac{d-1}d\}=1-2/3=1/3$. We thus obtain that 
$\Lambda_{\max}(f)=\{(1,1,0)\}$. 
\end{exam}
\begin{prop}\label{A}
For any integer $k\in \NN^*$, we \coGR{have} that 
\begin{align*}
e_f
=&\sup
\Bigl\{
1-\frac{s(a)}{m_f(a)}:a\in(\RR_+^*)^n, \ m_f(a)=k
\Bigr\}
\\
=&\max\Bigl\{1-\frac{s(a)}{m_f(a)}:a\in\Lambda(f),\ m_f(a)>0\Bigr\}. 
\end{align*}
In particular $e_f \in \QQ$. Moreover 
$$
e_f\ge\max
\Bigl\{
1-\frac{s(a)}{m_f(a)}:a\in(\NN^*)^n, \ m_f(a)=k
\Bigr\}.
$$
and the equality holds 
for infinitely many $k$ if and only if $\Lambda_{\max}(f)\cap (\NN^*)^n \neq \emptyset$. 
\end{prop}


\begin{proof}[Proof of Proposition \ref{A}]
Since the function $s$ is linear and the function $m_f$ is piecewise linear, 
then $\frac{s(\lambda a)}{m_f(\lambda a)}=\frac{s(a)}{m_f(a)}$ 
for any $\lambda \in \RR_+^*$ and $a\in (\RR_+^*)^n$. 
In particular, the right\coGR{-}hand side of the first equality 
does not depend on $k$. 
Then the equality is \coGR{clearly attained for infinitely many $k$} 
by definition of $e_f$.

To see the second equality, 
we choose $k=1$ and we take a sequence 
$(a^j)_{j\in \NN}\in \big((\RR_+^*)^n\big)^{\NN}$ with
$m_f(a^j)=1$ and $1-s(a^j)\to e_f$ as $j\to \infty$. 
The sequence $(s(a^j))_{j\in \NN}$ is bounded, 
therefore so is $(a^j)_{j\in \NN}$ and, taking a subsequence if necessary, 
we may assume that the sequence $(a^j)_{j\in \NN}$ is convergent to 
a limit $b\in \RR_+^n$ satisfying $m_f(b)=1$ and $e_f=1-s(b)$. 
We want to prove that there exist $\lambda\in \RR_+^*$ 
and a primitive generator $v\in \Lambda(f)$ with $m_f(\lambda v)=1$ 
and $s(\lambda v)=s(b)$. 
Assume $b$ is not such a multiple of an element in $\Lambda (f)$. 
Denote by $\sigma$ the subcone of $\Gamma^*(f)$ corresponding to 
the face $\gamma_f(b)$ of the Newton polyhedron of $f$, 
and consider the restriction of $s$ to $\sigma \cap \{m_f=1\}$. 
\coGR{The} restriction of the linear function $s$ \coGR{must} be constant 
otherwise $e_f=1-s(b)$ can not be the supremum of $1-s/m$. 
As a consequence, for any $v\in \Lambda (f)$ \coGR{with $v\in\sigma$}. 
and for $\lambda \in \RR_+^*$ with $m_f(\lambda v)=1$, 
we obtain $s(\lambda v)=s(b)$. 

The last inequality is clear from the first equality, and the condition on when the equality holds comes from the second equality.
\end{proof}

By definition, $e_f$ is strictly less than $1$. We describe its sign in next proposition. 

\begin{prop}\label{1111} The sign of $e_f$ is decided by the following conditions.
\begin{itemize}
\item If $(1,\dots,1)\not\in\Gamma_+(f)$, then $e_f>0$. 
\item If $(1,\dots,1)\in\Gamma(f)$, then $e_f=0$.  
\item If $(1,\dots,1)$ \coGR{is in} the interior of $\Gamma_+(f)$, then $e_f<0$.
\end{itemize}
\end{prop}

\begin{proof}
By definition of the function $m_f$, 
the following equivalences hold:
\begin{itemize}
\item 
$(1,\dots,1)\in\Gamma_+(f)$ if and only if $m_f(a)\le s(a)$ 
for any $a\in (\mathbb R_+^*)^n$. 
\item 
$(1,\dots,1)\in$ the interior of $\Gamma_+(f)$ 
if and only if $m_f(a)<s(a)$  for any 
$a\in (\mathbb R_+^*)^n$. 
\item 
$(1,\dots,1)\notin \Gamma_+(f)$ if and only if $m_f(a)>s(a)$  
for some $a\in (\mathbb R_+^*)^n$. 
\end{itemize}
\end{proof}

\coGR{To prepare for the forthcoming} proof of Theorem \ref{LinearBound}, 
we introduce some more notation corresponding to numbers $e_{\gamma}$ 
analog\coGR{ous} to $e_f$ but restricted to a given face $\gamma$ 
of the Newton polyhedron of $f$, and to elements $a\in (\RR_+^*)^n$ 
realising such $e_{\gamma}$.

\begin{defn}\label{def:E} 
Let $\gamma$ be a compact face of $\Gamma_+(f)$.
We set 
\begin{align*}
e_\gamma=&
\sup\{1-\tfrac{s(a)}{m_f(a)}:a\in(\RR_+^*)^n,\ \gamma(a)=\gamma\},\\ 
m_\gamma=&
\min\{m_f(a):a\in(\NN^*)^n, \ \gamma(a)=\gamma\}.
\end{align*} 
For $k \in \mathbb N^*$, we denote by $E_k$ the set of $n$-tuples of positive numbers 
realising $e_f$ on the level $m_f=k$,
$$
E_k=\Bigl\{a\in(\RR_+^*)^n: 
e_f=1-\frac{s(a)}{m_f(a)}, \ m_f(a)=k\Bigr\},
$$
\coGR{and} we denote by $E_k(\gamma)$ its version relative to the face 
$\gamma$\coGR{;} that is, 
$$
E_k(\gamma)=\Bigl\{a\in(\RR_+^*)^n: 
e_\gamma=1-\frac{s(a)}{m_f(a)},\ m_f(a)=k, \ \gamma(a)=\gamma
\Bigr\}.
$$
\coGR{Also, we define} $e^0_f$ by 
$$
e^0_f=\max\{m_\gamma e_\gamma:X_\gamma\ne\emptyset\}.
$$ 
Similarly in the case with sign, we set
$$e^\pm_f=\max\{e_\gamma:X^\pm_\gamma\ne\emptyset\}.$$
\end{defn}

\begin{rem} We have $e_\gamma=\max\{1-\frac{s(a)}{m_f(a)}:a\in\Lambda(f),\
\gamma(a)\supset \gamma\}$ as a consequence of the proof of Proposition \ref{A}. 
It is clear that $e_f=\max\{e_\gamma:\gamma<\Gamma_+(f)\}$. 
\end{rem}

For $k\in \mathbb N^*$ and $\gamma$ a compact face of $\Gamma_+(f)$, the set $E_k(\gamma)$ is defined by a linear equation in the cone corresponding to $\gamma$, so $E_k(\gamma)$ is convex. 
\coGR{In fact,} the same holds true for $E_k$.

\begin{lem}
For $k\in \mathbb N^*$, the set $E_k$ is convex.
\end{lem} 

\begin{proof}
Take $a, b\in E_k$, and for $t\in [0,1]$ we define $c=(1-t)a+tb$. We have $s(c)=k(1-e_f)$ by linearity of $s$ since $s(a)=s(b)=k(1-e_f)$. 
We want to show that $m_f(c)=k$. On one hand, we know that $m_f(c)\ge k$ by definition of $m_f$. On the other hand, since $m_f(\frac{k}{m_f(c)}c)=k$, we have 
$$
k(1-e_f)
\le s(\tfrac{k}{m_f(c)}c)$$
by definition of $e_f$. Moreover
$$s(\tfrac{k}{m_f(c)}c)
=\tfrac{k}{m_f(c)}s(c)
=\tfrac{k}{m_f(c)}k(1-e_f).
$$ 
and, \coGR{since $1-e_f>0$, we have} $m_f(c)\le k$. 
\coGR{Therefore,} $m_f(c)=k$ and $c$ belongs to $E_k$.
\end{proof}


\subsection{Linear bound for the degree}\label{Le}

Let $A_k(u)$ denote elements of $\ZZ[u][[u^{-1}]]$ for $k\in \NN^*$. 
We look for a linear bound for the degree of $A_k(u)$, namely a bound of the form
$$
\deg A_k(u)\le \alpha k+\beta,$$
with $(\alpha,\beta)\in \QQ \times \QQ$ such that \coGR{equality} 
holds for infinitely many $k\in \NN^*$.  
Note that if it exists, such a pair $(\alpha, \beta)$ is unique. Setting 
$$
Z(t)=\sum_{k\ge1}A_k(u)t^k,\quad 
A_k(u)=c_ku^{\alpha k+\beta}+(\textrm{lower order terms}), 
$$ 
where $c_k\in \ZZ$, we have 
$$
\frac{Z(u^{-\alpha} t)}{u^\beta}
=\sum_{k\ge1}\frac{A_k(u)}{u^{\alpha k+\beta}}t^k
\to\sum_{k\ge1}c_kt^k\qquad \textrm{as~~} u\to\infty. 
$$
So a characterisation of $(\alpha,\beta)\in \QQ \times \QQ$ is that 
it is the unique pair such that 
$$
\lim_{u\to\infty}\frac{Z(u^{-\alpha} t)}{u^\beta}
$$
is a non\coGR{-}zero series\coRR{.} 
For a given integer $k\in \NN^*$, we call $c_ku^{\alpha k+\beta}$ 
\textbf{the leading term} and $c_k$ the \textbf{leading coefficient} 
with respect to the linear bound of the degree of $A_k(u)$. 
We call $\alpha k+\beta$ the \textbf{leading exponent} of $Z(t)$.

In Theorem \ref{LinearBound}, we give a \coGR{bound} for the degree of 
the virtual Poincar\'e polynomials of the arc spaces 
$\mathcal A_k(f)$ and $\mathcal A^\pm_k(f)$
associated with a polynomial function $f$. Moreover, we give a condition so that a linear bound exists for $\beta(\mathcal A_k(f))$. We express these bounds in terms of the integral numbers $e_f$, introduced in \eqref{Def:Ef}, and $e^0_f$ and $e^\pm_f$, introduced in Definition \ref{def:E} .

\begin{thm}\label{LinearBound}
Let $f$ be a non-degenerate polynomial. 
For all $k \in \mathbb N^*$, the degree bound of the virtual Poincar\'e polynomial of $\mathcal A_k(f)$ 
is given by 
\begin{align*}
\deg\beta(\mathcal A_k(f))\le&
\begin{cases}
n-k+k e_f&\textrm{~~if~~}e_f>0,\\
n-k&\textrm{~~if~~}e_f=0,\\
n-k+\max\{k e_f, e^0_f\}&\textrm{~~if~~}e_f<0.
\end{cases}
\end{align*}
Moreover if 
\coRR{$e_f \geq 0$, these bounds give the leading exponent of $Z_f(t)$ 
if and only if $\Lambda_{\max}(f)\cap (\NN^*)^n \neq \emptyset$}.
In the case with sign, we have 
\begin{align*}
\deg\beta(\mathcal A^\pm_k(f))\le&
\begin{cases}
n-1-k+k e^\pm_f&\textrm{~~if~~}e^\pm_f>0,\\
n-1-k&\textrm{~~if~~}e^\pm_f=0,\\
n-1-k+\max\{k e^\pm_f, e^0_f\}&\textrm{~~if~~}e^\pm_f<0.
\end{cases}
\end{align*}
\end{thm}

\begin{proof}
The equalities in Theorem \ref{Thm1} imply that  
$$\deg\beta(\mathcal A_k(f))=\max\{p_k, \ q_k\},$$
where
\begin{align*}
p_k=&n+\max\{-s(a):a\in(\NN^*)^n, \ m_f(a)=k\}, \\
q_k=&1+\max\{q_k(\gamma):X_\gamma\ne\emptyset, \gamma < \Gamma(f)\}, \\
q_k(\gamma)=&
\dim X_\gamma+\max\{
-s(a)+m_f(a)-k:a\in(\NN^*)^n, \ m_f(a)<k, \ \gamma(a)=\gamma\}.
\end{align*}
Concerning the term $p_k$, we have 
\begin{align*}
p_k=&n-k+\max\{m_f(a)-s(a):a\in(\NN^*)^n,  \ m_f(a)=k\}\\
=&n-k+k
\max
\bigl\{
1-\tfrac{s(a)}{m_f(a)}:a\in(\NN^*)^n,  \ m_f(a)=k
\bigr\}\\
\le&n-k+ke_f,
\end{align*}
and \coGR{equality} holds for infinitely many $k$ 
if and only if $\Lambda_{\max}(f)\cap (\NN^*)^n\neq \emptyset$\coGR{,} 
by Proposition \ref{A}. 
Concerning the term $q_k$, consider a face $\gamma<\Gamma(f)$. Then
\begin{align*}
q_k(\gamma)=&
\dim X_\gamma-k+\sup\{m_f(a)-s(a):a\in(\NN^*)^n,\ m_f(a)<k,\ \gamma(a)=\gamma\}\\
=&
\dim X_\gamma-k+\sup
\{
m_f(a)\bigl(1-\tfrac{s(a)}{m_f(a)}\bigr):
a\in(\NN^*)^n,m_f(a)<k, \gamma(a)=\gamma \}.
\end{align*}
Here we remark that if $e_f>0$, then 
$$
q_k(\gamma)
<n-1-k+ke_\gamma\coGR{,}
$$
\coGR{whilst if} $e_f=0$, then 
$$
q_k(\gamma)\leq n-1-k.
$$
\coGR{Finally, in the } case $e_f \leq 0$, we obtain 
$$
q_k(\gamma)\le n-1-k+m_\gamma e_\gamma,
$$
and the inequality is not optimal in general. Finally for $q_k$ 
the \coGR{bound} is
$$
q_k \leq \begin{cases}
n-k-1+ke_f& \textrm{ if }e_f>0\\
n-k& \textrm{ if }e_f=0\\
n-k+\max\{m_\gamma e_\gamma:X_\gamma\ne\emptyset\}&\textrm{ if }e_f\le0.
\end{cases}
$$ 
As a consequence
$$
\deg\beta(\mathcal A_k(f))=\sup\{p_k,q_k\}
\le\begin{cases}
n-k+ke_f& \textrm{ if }e_f\geq 0\\
n-k+\max\{ke_f,
m_\gamma e_\gamma:X_\gamma\ne\emptyset\}&\textrm{ if }e_f<0.
\end{cases}
$$
In the case with sign, we obtain similarly by Theorem \ref{Thm1}
$$\deg\beta(\mathcal A^\pm_k(f))=\sup\{
p^\pm_k(\gamma):X^\pm_\gamma\ne\emptyset, \  
q_k(\gamma):X_\gamma\ne\emptyset; \gamma < \Gamma_+(f)\}$$
where 
$$p^\pm_k(\gamma)=
\dim X_\gamma^{\pm}+\sup\{
-s(a): a\in(\NN^*)^n,\ m_f(a)=k,\ \gamma(a)=\gamma\}.$$
Then
\begin{align*}
p^\pm_k(\gamma)
=&
\dim X^\pm_\gamma-k+
\sup\{m_f(a)-s(a):a\in(\NN^*)^n,\ m_f(a)=k,\ \gamma(a)=\gamma\}
\\
=&
\dim X^\pm_\gamma-k+k\sup
\bigl\{
1-\tfrac{s(a)}{m_f(a)}:a\in(\NN^*)^n, \ m_f(a)=k, \ \gamma(a)=\gamma
\bigr\}
\\ 
\le&\dim X^\pm_\gamma-k+ke_\gamma,
\end{align*} 
and \coGR{equality} is attained \coGR{whenever} 
$E_k(\gamma)\cap (\NN^*)^n \neq \emptyset$. 
As a consequence
\begin{align*}
&\deg\beta(\mathcal A^\pm_k(f))=
\sup\{p^\pm_k(\gamma):X_\gamma^\pm\ne\emptyset,\ 
q_k(\gamma):X_\gamma\ne\emptyset\}
\\
&\le
\begin{cases}
n-1-k+k\max\{e_\gamma:X^\pm_\gamma\ne\emptyset\}&
\textrm{ if }e^\pm_f>0\\
n-1-k+\max\{ke_\gamma: X_\gamma^{\pm}\ne\emptyset, \ 
m_\gamma e_\gamma:X_\gamma\ne\emptyset\}&
\textrm{ if }e^\pm_f\le0, 
\end{cases}
\end{align*}
and the proof is completed.
\end{proof}

In the cases \coRR{that the bounds above do not give the} leading exponent, 
we \coGR{conclude some} information on the degree on the arc spaces. 
\coGR{In fact}, we have the following results.

\begin{prop}\label{prop-posi} If $e_f\geq 0$, then 
$$
e_f=1+\limsup_{k\to +\infty} \frac{\deg \beta(\mathcal A_k)}{k}.
$$
\end{prop}

\begin{proof} 
The result follows directly from Theorem \ref{LinearBound} if $\Lambda_{\max}(f)\cap (\NN^*)^n \neq \emptyset$. 
\coGR{Otherwise}, we know that for $k\in \NN^*$ and $\epsilon >0$, 
there exist $a\in (\QQ^*)^n$ with 
$$
0\leq e_f-(1-\frac{s(a)}{m_f(a)})\leq \epsilon
$$
by definition of $e_f$. 
In particular, there exists \coGR{sufficientlly large} $k_{\epsilon}$ 
such that for $k\in \NN^*$ \coGR{a} multiple of $k_{\epsilon}$, 
there exist $a\in (\NN^*)^n$ such that the inequality above 
is satisfied together with $m_f(a)=k$.  
As a consequence, using the notation introduced in the proof of 
Theorem \ref{LinearBound}, we have
$$
p_k\geq n-k+k(e_f-\epsilon)
$$
for $k$ \coGR{a} multiple of $k_{\epsilon}$.
This implies that
$$n-k+k(e_f-\epsilon)\leq \deg \beta(\mathcal A_k)\leq n-k+ke_f,$$
and thus
$$ e_f-\epsilon\leq \frac{\deg \beta(\mathcal A_k)-n+k}{k} \leq e_f$$
for $k$ \coGR{a} multiple of $k_{\epsilon}$.
\end{proof}

\coRR{As a consequence of Theorem \ref{LinearBound} and Proposition \ref{prop-posi}, we obtain the following result.}
\begin{thm} \coRR{Let $f$ be a finitely determined non-degenerate polynomial with $e_f\geq 0$. Then 
$$\deg \beta(\mathcal A_k(f)) \leq n-k+ke_f$$
and the bound is attained by infinitely many $k\in \mathbb N$.}
\end{thm}

\begin{proof} \coRR{Since $f$ is finitely determined, then $f$ and $g=f+\sum_{i=1}^n x_i^d$ are right-equivalent for $d\in \mathbb N$ big enough. In particular $f$ and $g$ share the same zeta functions, and by Proposition \ref{prop-posi} we obtain $e_f=e_g$. Since $g$ is convenient, we know that $\Lambda_{\max}(g)\cap (\mathbb N^*)^n \neq \emptyset$ and therefore the linear bound
$$\deg \beta(\mathcal A_k(g)) \leq n-k+ke_g$$
is attained for infinitely many $k\in \mathbb N$ by Theorem \ref{LinearBound}. This means that the linear bound 
$$\deg \beta(\mathcal A_k(f)) \leq n-k+ke_f$$
is attained by infinitely many $k\in \mathbb N$. }
\end{proof}

In \coGR{the} case $e_f\leq 0$, this superior limit in Theorem \ref{LinearBound}
\coGR{also} \coGR{provides} some information.

\begin{prop}\label{prop-nega} In the case $e_f< 0$, we have
$$1+\limsup_{k\to +\infty} \frac{\deg \beta(\mathcal A_k)}{k} \in [\max \{e_{\gamma}: X_{\gamma} \neq \emptyset\},0],$$
and it is equal to $e_f$ if $\{f=0\}=\{0\}$.
\end{prop}

\begin{proof} If $\{f=0\}=\{0\}$, 
then the degree of the virtual Poincar\'e polynomial of $\mathcal A_k(f)$ 
is given by $p_k$, 
using the notation introduced in the proof of Theorem \ref{LinearBound}, 
and the same proof as that of Proposition \ref{prop-posi} gives the result.

If there exists a face $\gamma$ with $X_{\gamma}\neq \emptyset$, then
$$\frac{\deg \beta(\mathcal A_k)-n+k}{k} \leq \frac{\max \{ke_f,e^0_f \}}{k}$$
by Theorem \ref{LinearBound}. Moreover the right\coGR{-}hand side is bounded by
 $e^0_f/k$ for $k$ \coGR{sufficiently large} 
since $e_f<0$, so that the superior limit is less than or equal to zero. 
Now, choose $\epsilon>0$. 
Similarly to the proof of Proposition \ref{prop-posi}, 
there exist $k_{\epsilon}\in \NN^*$ and $a\in (\NN^*)^n$ such that 
$$
0\leq e_f-(1-\frac{s(a)}{m_f(a)})\leq \epsilon
$$
and $m_f(a)=k_{\epsilon}$. In particular, for any $k$ equal to one plus a multiple of $k_{\epsilon}$, we have
$$q_k(\gamma)\geq \dim X_{\gamma} -k+(k-1)(e_{\gamma}-\epsilon).$$
As a consequence, for such $k$ we have
$$\frac{1+\dim X_{\gamma}-n}{k}+\frac{k-1}{k}(e_{\gamma}-\epsilon)\leq \frac{\deg \beta(\mathcal A_k)-n+k}{k}$$
since $\deg \beta(\mathcal A_k)\geq q_k\geq 1+q_k(\gamma)$. We thus obtain that
$$e_{\gamma}-\epsilon \leq \limsup_{k\to +\infty} \frac{\deg \beta(\mathcal A_k)-n+k}{k},$$
and this concludes the proof.
\end{proof}

Next corollary will be useful in the sequel.

\begin{cor}\label{corA} 
Assume that $e_f\geq 0$ and $\Lambda_{\max}(f)\cap (\NN^*)^n$ is 
\coGR{a set consisting of} one  point $v$. Then  
$$
\lim_{u\to+\infty}\frac{Z_f(u^{1-e_f}t)}{u^n}=
\frac{t^{m_f(v)}}{1-t^{m_f(v)}}.
$$
\end{cor}

\begin{proof} By Theorem \ref{LinearBound}, the zeta function of $f$ admits 
as leading exponent $(e_f-1)k+n$, and the bound is attained 
for all multiple\coGR{s} of $m_f(v)$. 
In this case the leading coefficient is equal to one, and this implies the result.
\end{proof}

A facet is an ($n-1$)-dimensional face of $\Gamma_+(f)$. Let $\gamma$ be a compact facet of $\Gamma_+(f)$. The associated cone of the subdivision of $\Gamma^*(f)$ is one dimensional, and we denote by $v$ in this section the primitive vector with 
$\gamma(v)=\gamma$.
Recall that $h=m_f-s$.

\begin{lem} 
For $k\in \NN^*$, the degree of $P_k(\gamma)$ is given by 
\begin{align*}
\deg P_k(\gamma)=&
\begin{cases}
-\infty&\textrm{if }\ m_f(v)\nmid\ k,\\
-k\frac{s(v)}{m_f(v)}&\textrm{if }\ m_f(v)\mid k
\end{cases}
\end{align*}
and the degree of $Q_k(\gamma)$ by
\begin{align*}
\deg Q_k(\gamma)=&
\begin{cases}
-\infty&\textrm{if }\ k\le m_f(v),\\
-k+\bigl\lfloor\frac{k-1}{m_f(v)}\bigr\rfloor
(m_f(v)-s(v))&\textrm{if }\  k>m_f(v) \ \textrm{ and }\  h(v)>0,\\
-k+m_f(v)-s(v)&\textrm{if }\ k>m_f(v)\ \textrm{ and }\ h(v)\leq 0.
\end{cases}
\end{align*}
\end{lem}

\begin{proof}
The proof is a direct consequence of the following equalities: 
\begin{align*}
\deg P_k(\gamma)=&\max\{-s(a):a=\lambda v\, \lambda \in \NN^*,\ m_f(a)=k\}\\
\deg Q_k(\gamma)=&-k+\max\{m_f(a)-s(a):a=\lambda v, \lambda
\in\NN^*, \ m_f(\lambda v)\le k-1\},
\end{align*}
coming from the definition of $P_k(\gamma)$ and $Q_k(\gamma)$.
\end{proof}

In particular, we note that for $k=m_f(v)$ we have $\deg P_k(\gamma)=-s(v)$ whereas $Q_{k}(\gamma)=0$ .

\begin{cor} Let $k\in \NN^*$. Assume that $k$ is divisible by
  $m_f(v)$. Then the sign of $h(v)$ 
\coGR{determines} the following: 
\begin{itemize}
\item if $h(v)>0$, then $\deg P_k(\gamma)>\deg Q_k(\gamma)$, 
\item if $h(v)=0$, then $\deg P_k(\gamma)=\deg Q_k(\gamma)$, 
\item if $h(v)<0$ and $k\neq m_f(v)$, then $\deg P_k(\gamma)<\deg Q_k(\gamma)$. 
\end{itemize}
\end{cor}

\begin{proof}
If $m_f(v)-s(v)>0$ and $m_f(v)\mid k$, we obtain 
\begin{align*}
\deg P_k(\gamma)-\deg Q_k(\gamma)
=&
-k\frac{s(v)}{m_f(v)}+k-
\Bigl\lfloor\frac{k-1}{m_f(v)}\Bigr\rfloor(m_f(v)-s(v))\\
=&
\biggl(\frac{k}{m_f(v)}-\Bigl\lfloor\frac{k-1}{m_f(v)}\Bigr\rfloor\biggr)(m_f(v)-s(v))
\coRR{>0,}
\end{align*}
\coRR{and} this proves the first statement. 
If $m_f(v)-s(v)\le0$ and $m_f(v)\mid k$, we obtain 
\begin{align*}
\deg P_k(\gamma)-\deg Q_k(\gamma)
=&
-k\frac{s(v)}{m_f(v)}+k-(m_f(v)-s(v))
\\
=&
\biggl(\frac{k}{m_f(v)}-1\biggr)(m_f(v)-s(v)),
\end{align*}
and this implies the remaining cases.
\end{proof}

\subsection{Convenient weighted homogeneous polynomials}\label{weight}

Let $f\in \RR[x_1,\ldots,x_n]$ be a weighted homogeneous polynomial with respect to the weight
$(w_1,\dots,w_n;d)$. We assume that $f$ is convenient. 

In this case $\Gamma_+ (f)$ has a unique compact facet $\gamma_f$, 
and its associated $1$-cone is generated by the primitive vector 
$v=(w_1,\ldots,w_n)$, with $m_f(v)=\coRR{d}$. 

Note that $\Lambda_{\max}(f)=\{v\}\in (\NN^*)^n$.
As a consequence,  
$$
h(v)=m_f(v)-\sum_{i=1}^n w_i
$$ 
and 
$$
e_f=\frac{h(v)}{m_f(v)}=1-\frac{1}{m_f(v)}\sum_{i=1}^n {w_i}.
$$
In particular, if we
are able to compute $h(v)$ and $m_f(v)$ from the zeta function, 
then we can recover the sum of the weights of $f$.

Assume that the zeta function of $f$ is given. Then we can decide whether $e_f>0$ 
of $e_f\leq 0$ using Proposition \ref{prop-posi} and Proposition \ref{prop-nega}. 
If $e_f>0$, we \coGR{recover} the value of $e_f$ by Proposition \ref{prop-posi},
 together with $m_f(v)$ by Corollary \ref{corA}. 
In this case we derive the value of $h(v)$ \coGR{as} $h(v)=m_f(v)e_f$.

In case $e_f \leq 0$, the situation is more subtle to analyse. Note
that if $X_{\gamma_f}$ is not empty, then the degree of the virtual Poincar\'e polynomial 
of $\mathcal A_k(f)$ is given by 
$$
n+\max \{\deg P_k(\gamma_f),\deg Q_k(\gamma_f)\}
\}
$$ 
--- but it may be less if $X_{\gamma_f}$ is empty. 
The degree of $P_k(\gamma_f)$ and $Q_k(\gamma_f)$ may be expressed by the formulae
$$
\deg P_k(\gamma_f)
=-k+\max \{h(a):~a\in (\mathbb N^*)^n,~m_f(a)=k,~\gamma(a)=\gamma_f \}
$$
and 
$$
\deg Q_k(\gamma_f)
=-k+\max \{h(a):~a\in (\mathbb N^*)^n,~m_f(a)<k,~\gamma(a)=\gamma_f \}.
$$
Therefore we are \coGR{led} to analyse the levels of the
function $h=m_f-s$ on $\mathbb N^n$, and more precisely on the subsets
of $\mathbb N^n$ defined by $\{m_f(a)=k\}$ and $\{m_f(a)<k\}$.

To begin with, let us forget that we are interested in integral points
and describe its
levels on $\mathbb R_+^n$. The function $h$ is linear on each
cone of the subdivision of $\Gamma^*(f)$, therefore its levels are completely described
by its value on $v$ and on the canonical basis $\{e_1,\ldots,e_n\}$ of
$\mathbb R^n$. Note that $h(e_i)=m_f(e_i)-1= -1$ since the
Newton polyhedron is convenient.

In particular, the levels of $h$ on $\mathbb R_+^n$ are described as follows:
\begin{itemize}
\item if $h(v)=0$, there are only non-positive levels
that are cylinders parallel to the line generated by $v$, 
\item if $h(v)<0$, there are only strictly negative levels
that are unions of $n$ bounded simplexes with a
vertex on the line generated by $v$ and other vertexes on $(n-1)$ 
positive coordinates axis, whereas 
\item if $h(v)>0$, the levels of $h$ are unions of $n$ unbounded 
simplexes with a vertex on the line generated by $v$.
\end{itemize}

\coGR{Returning} to the computation of the degrees of 
$P_k(\gamma)$ and $Q_k(\gamma)$, 
we need to investigate the integral points on these
levels.

\begin{lem} 
Let $k\in \NN^*$. Assume that $e_f \leq 0$ and $X_{\gamma_f}\neq \emptyset$. 
\begin{itemize}
\item $\deg \beta(\mathcal A_k(f))\leq n-k$ with equality 
for \coGR{infinitely many $k$} if and only if $h(v)=0$,
\item $\deg \beta(\mathcal A_k(f))<n-k$ if and only if $h(v)<0$. 
In th\coGR{is} case, for \coGR{infinitely many $k$}, we have
$$
\deg \beta(\mathcal A_k(f))=n-k+
\max\{h(a):a\in (\NN^*)^n,\ X_{\gamma(a) }\ne \emptyset \}.
$$
\end{itemize}
\end{lem}

\begin{proof} 
It suffices to compute the maximum of $h$ 
on $\{a\in(\RR_+)^n:m_f(a)=k\}$ and 
on $\{a\in(\RR_+)^n:m_f(a)<k\}$. 

If $h(v)=0$, then $h$ has non\coGR{-}positive levels. 
The maximum of $h$ on $\{a\in(\RR_+)^n:m_f(a)=k\}$ 
is attained on the line generated by $v$ (in general not at an integral point). 
In particular, if $k>m_f(v)$, the maximum of $h$ on $\{a\in (\RR_+)^n:m_f(a)<k\}$ 
is attained at $v$.

In \coGR{the} case $h(v)<0$, then $h$ has strictly negative levels. If $k>m_f(v)$, 
then 
$$
h(v)\leq \max\{h(a):a\in (\NN^*)^n,\ X_{\gamma(a) }\ne \emptyset \}<0.
$$
\end{proof}

\begin{rem} 
It may happen that $\max \{h(a):~a\in (\mathbb N^*)^n\}$ 
is strictly bigger than $h(v)$ in the case $h(v)<0$. 
Consider for example $f(x_1,x_2,x_3)=x_1^2-x_2^2+x_3^m$ with $m$ odd. 
Then $v=(m,m,2)$ and $m_f(v)=2m$ so that $h(v)=m_f(v)-s(v)=-2$. 
\coGR{However,} for $a=(1,1,1)$, we have $h(a)=2-3=-1$ and 
$\gamma(a)$ is the face corresponding to $x_1^2-x_2^2$\coGR{,} 
so that $X_{\gamma(a) }\ne \emptyset$. 
\end{rem}

Assume that $e_f \leq 0$ and $X_{\gamma_f}=\emptyset$. 
Then  $X_{\gamma}=\emptyset$ for
all face\coGR{s} $\gamma$ of $\Gamma (f)$, therefore the degree of the virtual
Poincar\'e polynomial of $\mathcal A_k(f)$ is simply
given by the terms $P_k(\gamma)$ for $\gamma < \Gamma(f)$, and the discussion becomes easier. 

\begin{lem} Assume that $e_f \leq 0$ and $X_{\gamma_f}= \emptyset$. Let $k\in
  \NN^*$. Then
\begin{itemize}
\item $\deg \beta(\mathcal A_k(f))\leq n-k$ with equality for infinitely
  many $k$ if and only if $h(v)=0$,
\item $\deg \beta(\mathcal A_k(f))<n-k$ if and only if $h(v)<0$.
\end{itemize}
\end{lem}

\begin{proof} The degree of $\beta(\mathcal A_k(f))$ is given by 
$$\max\{-s(a): ~a\in (\mathbb N^*)^n,~m_f(a)=k \}$$
because the terms of the form $Q_k(\gamma)$ vanish, for $\gamma < \Gamma(f)$.
Moreover
$$ \sup\{-s(a): ~a\in (\mathbb R_+)^n,~m_f(a)=k \}=
-s\Bigl(\frac{k}{m_f(v)}v\Bigr),
$$
since the levels of $m_f$ on $\RR_+^n$ are 
the boundary of cones \coGR{with vertex a point lying on the line generated by $v$,
whose} sides are parallel to the coordinate hyperplanes. As a consequence
$$ \max\{-s(a): ~a\in (\mathbb N^*)^n,~m_f(a)=k \}\leq
-s\Bigl(\frac{k}{m_f(v)}v\Bigr),
$$
and the equality holds when $m_f(v)$ divides $k$. In particular, 
$s(\frac{k}{m_f(v)}v)=k$ if $h(v)=0$, whereas $s(\frac{k}{m_f(v)}v)>k$ if $h(v)<0$.
\end{proof}

Therefore we are able to recognise the sign of $e_f$, and thus the sign of $h(v)=m_f(v)e_f$, from the zeta
function of $f$ via Proposition \ref{prop-posi} and Proposition \ref{prop-nega}. Moreover we are able to obtain the value of $m_f(v)$, and thus of $h(v)$, when we already know that $h(v)$ is
greater than or equal to zero thanks to Corollary \ref{corA}. We collect the different possibilities for the sign of $h(v)$ in next proposition.

\begin{prop}\label{signh} 
Assume that $f$ is a convenient weighted homogeneous polynomial
\coGR{which is} non\coGR{-}degenerate with respect to its Newton polyhedron. 
Denote by $v$ the primitive vector associated with $f$. 
\begin{itemize}
\item $h(v)>0$ if and only if $e_f>0$, and more precisely $h(v)=m_f(v)e_f$.
\item $h(v)=0$ if and only if $\deg \beta(\mathcal A_k(f))\leq n-k$, with equality 
for infinitely many $k\in \NN^*$.
\item $h(v)<0$ if and only if $\deg \beta(\mathcal A_k(f))<n-k$ 
for \coRR{all} $k\in \NN^*$.
\end{itemize}
\end{prop}

\begin{rem}\label{2-case} 
In \cite{Ould}, it is shown that the weights of a non-degenerate weighted homogeneous 
polynomial in two \coGR{variable} are invariants under blow-analytic equivalence, 
using the zeta function defined with the Euler characteristic of the homology of 
locally finite chains with closed supports \cite{KP}. 
Because of the good properties of the virtual Poincar\'e polynomial, 
we recover easily the analogous result, in the setting of blow-Nash
 equivalence. 
Indeed, the first exponent of the zeta function combined with 
the leading exponent give the weights. 
We mention moreover that, in the two \coGR{variable} case, \coGR{a} complete
classification has been recently achieved by S. Koike and A. Parusi\'nski \cite{KP10}.
\end{rem}

\section{Recovering the weights}\label{reco}

We focus in this section on how to recover the weights of a convenient non-degenerate 
weighted homogeneous polynomials in three variables from its zeta function.  
In the two \coGR{variable} case, the result is immediate by Proposition \ref{signh}, since $h(v)\geq 0$ as soon as the germ is singular.

Let $f(x_1,x_2,x_3)\in \RR[x_1,x_2,x_3]$ be a weighted homogeneous polynomial whose Newton polyhedron is convenient. 
Denote by $w_1,w_2,w_3\in \NN$ the weights and by $d\in \NN$ the weighted degree of $f$. 
Recall that $v=(w_1,w_2,w_3)$ is the primitive generator of the 1-cone associated with $\Gamma_+(f)$, and that $m_f(v)=d$.
Let $V_1=(p_1,0,0)$, $V_2=(0,p_2,0)$ and $V_3=(0,0,p_3)$ denote the vertexes of
$\Gamma_+(f)$, where $p_1,p_2,p_3\in \NN^*$ since $\Gamma_+(f)$ is convenient. Then $w_i=m_f(v)/p_i$ for $i=1,2,3$.
Assume that $p_1 \leq p_2 \leq p_3$ 
without loss of generality.

Let $\gamma_f$ denote the compact 2-dimensional face of $\Gamma_+(f)$.
Set $\gamma_i=\gamma_f\cap\{\nu_i=0\}$ for $i\in\{1,2,3\}$, so that the compact faces of the Newton polyhedron of $f$ are $\gamma_f$ in dimension two, $\gamma_1,\gamma_2, \gamma_3$ in dimension one and $V_1,V_2,V_3$ for the vertexes. 
Set $p_{ij}=\lcm(p_i,p_j)$ for $i\neq j \in\{1,2,3\}$ and recall that
$d=m_f(v)=\lcm(p_1,p_2,p_3)$.

We are going to give a complete description of the $k$-th coefficient of the zeta function of $f$, for $k\in \NN^*$. Theorem \ref{Thm1} describes these coefficients with respect to the contribution of each compact face $\gamma<\Gamma(f)$ of the Newton polyhedron of $f$, via the terms $P_k(\gamma)$ and $Q_k(\gamma)$ (understood as virtual Poincar\'e polynomials, cf. the introduction of section \ref{estim-deg}) given by
$$
P_k(\gamma)=\medsum_{a\in (\mathbb N^*)^3:\ m_f(a)=k,~\gamma(a)=\gamma}\,u^{-s(a)}
$$
and 
$$
Q_k(\gamma)=\medsum_{a\in (\mathbb N^*)^3:\ m_f(a)<k,~\gamma(a)=\gamma}\,u^{-k+m_f(a)-s(a)}.
$$ 
We begin with the terms of the form $P_k(\gamma)$.

\begin{lem}\label{lemP} Let $k\in \NN^*$ and $i\in \{1,2,3\}$. Set $\{j,s\}=\{1,2,3\}\setminus \{i\}$.
\begin{enumerate}
\item The term $P_k(\gamma_f)$ is non\coGR{-}zero if and only if $d|k$, and in this case $$P_k(\gamma_f)=u^{-(\frac{k}{p_1}
+\frac{k}{p_2}
+\frac{k}{p_3})}.$$
\item The term $P_k(\gamma_i)$ is non\coGR{-}zero if and only if $p_{js}|k$, and in this case $$P_k(\gamma_i)=(u-1)^{-1}u^{-(\lfloor\frac{k}{p_i}\rfloor
+\frac{k}{p_j}
+\frac{k}{p_s})}.$$
\item  The term $P_k(V_i)$ is non\coGR{-}zero if and only if $p_i|k$, and in this case $$P_k(V_i)=(u-1)^{-2}u^{-(\frac{k}{p_i}
+\lfloor\frac{k}{p_j}\rfloor
+\lfloor\frac{k}{p_s}\rfloor)}.$$
\end{enumerate}
\end{lem}

\begin{proof} 
The face $\gamma<\Gamma(f)$ gives a non\coGR{-}zero contribution $P_k(\gamma)$ 
if and only if the set $\{a\in (\NN^*)^n: m_f(a)=k, \gamma(a)=\gamma \}$ is not empty. 
Concerning $\gamma_f$, this set is equal to

\begin{align*}
\{\lambda v: \lambda \in \NN,~m_f(\lambda v)=k\}=&
\begin{cases}
\{\frac{k}{m_f(v)}v\}&\textrm{if }\ m_f(v)=d|k,\\
\emptyset &\textrm{otherwise,}
\end{cases}
\end{align*}
and therefore if $d|k$ we have $P_k(\gamma_f)=u^{-s(\frac{k}{m_f(v)}v)}$\coGR{,} 
where $v=(\frac{m_f(v)}{p_1},\frac{m_f(v)}{p_2},\frac{m_f(v)}{p_3})$. 
In the case of $\gamma_i$, this set is equal to
\begin{align*}
\{(a_i,a_j,a_s)\in (\NN^*)^3: ~a_ip_i>a_jp_j=a_sp_s=k\}=&
\begin{cases}
\{(a_i,\frac{k}{p_j},\frac{k}{p_s}):~a_i \geq \lfloor\frac{k}{p_i}\rfloor+1\}&\textrm{if }\ p_{js}|k,\\
\emptyset &\textrm{otherwise,}
\end{cases}
\end{align*}
and therefore 
$$
P_k(\gamma_i)=\sum_{l\geq 0}u^{-(\lfloor\frac{k}{p_i}\rfloor+1+l+\frac{k}{p_j}+\frac{k}{p_s})}=u^{-(\lfloor\frac{k}{p_i}\rfloor+1+\frac{k}{p_j}+\frac{k}{p_s})}\frac{u}{u-1}.
$$
Finally\coGR{,} in the case of $V_i$, this set is equal to
$$
\{(a_i,a_j,a_s)\in (\NN^*)^3: ~a_ip_i=k,~k<a_jp_j,~k<a_sp_s\}
$$
\begin{align*}
=&
\begin{cases}
\{(\frac{k}{p_i},a_j,a_s):~a_j \geq \lfloor\frac{k}{p_j}\rfloor+1,~a_s \geq \lfloor\frac{k}{p_s}\rfloor+1\}&\textrm{if }\ p_{i}|k,\\
\emptyset &\textrm{otherwise,}
\end{cases}
\end{align*}
and therefore 
$$
P_k(V_i)=\sum_{l_1,l_2\geq 0}u^{-(\frac{k}{p_i}+\lfloor\frac{k}{p_j}\rfloor+1+l_1+\lfloor\frac{k}{p_s}\rfloor+1+l_2)}=u^{-(\frac{k}{p_i}+\lfloor\frac{k}{p_j}\rfloor+1+\lfloor\frac{k}{p_s}\rfloor+1)}\big(\frac{u}{u-1}\big)^2.
$$
\end{proof}

Now we focus on the terms of the form $Q_k(\gamma)$.

\begin{lem}\label{lemQ} For $k\in \NN^*$ and $\gamma<\Gamma(f)$, we have
$$Q_k(\gamma)=\sum_{1\leq l<k}u^{-k+l}P_l(\gamma).$$
\end{lem}

\begin{proof} By definition of $Q_k(\gamma)$ we have
$$Q_k(\gamma)=\sum_{1\leq l<k} ~~\sum_{a\in (\mathbb N^*)^3:\ m_f(a)=l,~\gamma(a)=\gamma}\,u^{-k+m_f(a)-s(a)}=\sum_{1\leq l<k} u^{-k+l}P_l(\gamma).$$ 
\end{proof}

\begin{prop}\label{prop-coeff} 
Let $k\in \NN^*$. 
The coefficient of $t^k$ in 
the zeta function of $f$ is given by 
$$
\beta(\mathcal A_k(f))
=
\frac{(u-1)\xi_k}
{u^{
\lfloor\frac{k}{p_1}\rfloor
+\lfloor\frac{k}{p_2}\rfloor
+\lfloor\frac{k}{p_3}\rfloor}
}
+
\medsum_{1\le l<k}
\frac{(u-1)^2\eta_l}
{u^{k-l
+\lfloor\frac{l}{p_1}\rfloor
+\lfloor\frac{l}{p_2}\rfloor
+\lfloor\frac{l}{p_3}\rfloor
}}
$$
where 
\begin{align*}
\xi_k=&
\begin{cases}
0&\textrm{if }\ p_i\nmid k\ (i=1,2,3)\\
1&\textrm{if }\ p_i\mid k,\ p_{ij}\nmid k\ (i\ne j)\\
1+u-\beta(\widehat{X}_{\gamma_s})
&\textrm{if }\ p_{ij}\mid k, \ d \nmid k, \ \{i,j,s\}=\{1,2,3\}\\
1+u+u^2-\beta(\widehat{X}_{\gamma_f}\cup \widehat X_{\gamma_1}\cup \widehat X_{\gamma_2}\cup \widehat X_{\gamma_3})
&\textrm{if }\ d\mid k
\end{cases}
\end{align*}
and
\begin{align*}
\eta_l=&
\begin{cases}
0 &\textrm{\small if $p_i\nmid l$ $(i=1,2,3)$
or $p_i\mid l$, $p_{ij}\nmid l$ $(i\ne j)$}\\
\beta(\widehat{X}_{\gamma_s})&
\textrm{if $p_{ij}\mid l$, $d \nmid l$, $\{i,j,s\}=\{1,2,3\}$}\\
\textrm{\small $\beta(\widehat{X}_{\gamma_f}\cup \widehat X_{\gamma_1}\cup \widehat X_{\gamma_2}\cup \widehat X_{\gamma_3})$}
&\textrm{if $d\mid l$}.
\end{cases}
\end{align*}
\end{prop}

\begin{proof} For a compact face $\gamma$ of 
$\Gamma_+(f)$, we know from Remark \ref{set-intro} that  $$\beta(X_{\gamma})=(u-1)^{3-\dim \gamma}\beta(\widehat X_{\gamma}).$$
Moreover, if $\gamma=V_i$ is a vertex, where $i=1,2,3$, then $\beta(X_{V_i})=0$.
Combined with the formula given \coGR{in} Theorem \ref{Thm1}, 
we obtain that $\beta(\mathcal A_k(f))$ is equal to the sum of 
the contribution of the compact facet $\gamma_f$
$$
(u-1)\big((u-1)^2-\beta(\widehat X_{\gamma_f})\big)P_k(\gamma_f)
+(u-1)^2\beta(\widehat X_{\gamma_f})Q_k(\gamma_f),
$$
with the contribution of the $1$-dimensional faces $\gamma_1,\gamma_2,\gamma_3$,
$$
(u-1)^2\sum_{i=1}^3 \big((u-1)-\beta(\widehat X_{\gamma_i})\big)P_k(\gamma_i) +(u-1)^3\sum_{i=1}^3\beta(\widehat X_{\gamma_i})Q_k(\gamma_i),
$$
\coGR{and} with the contribution of the vertexes $P_1,P_2,P_3$\coGR{,} 
$$
(u-1)^3\sum_{i=1}^3 P_k(V_i).
$$
It remains to describe each contribution with respect to the relations 
between $p_1,p_2,p_3$ and $k$. 
We begin with $\xi_k$. By Lemma \ref{lemP}, it is clear that $\xi_k=0$ 
if none of the $p_i$ divides $k$ for $i=1,2,3$. 
If $p_i\mid k$ but $ p_{ij}\nmid k$ for $i\ne j$, 
then the only face contributing to $\xi_k$ is $V_i$ and by Lemma \ref{lemP}
$$
(u-1)^3P_k(V_i)=(u-1)u^{-(\frac{k}{p_i}
+\lfloor\frac{k}{p_j}\rfloor
+\lfloor\frac{k}{p_s}\rfloor)}
$$
so that $\xi_k=1$ as \coGR{claimed}. 
Assume that $p_{ij}\mid k$ but $d \nmid k$ and denote $\{i,j,s\}=\{1,2,3\}$. 
Then the contribution to $\xi_k$ comes from the faces $V_i,V_j$ and $\gamma_s$, 
and this contribution is equal to
$$(u-1)^2\big((u-1)-\beta(\widehat X_{\gamma_s})\big)P_k(\gamma_s)+(u-1)^3P_k(V_i)+(u-1)^3P_k(V_j)$$
$$=u^{-(\frac{k}{p_i}
+\frac{k}{p_j}
+\lfloor\frac{k}{p_s}\rfloor)}\big((u-1)((u-1)-\beta(\widehat X_{\gamma_s}))+2(u-1) \big)$$
by Lemma \ref{lemP}, so that $\xi_k=1+u-\beta(\widehat X_{\gamma_s})$ as expected. 
Finally, assume that $d\mid k$. Then all the faces contribute to $\xi_k$ and the contribution is equal to
$$(u-1)u^{-(\frac{k}{p_i}+\frac{k}{p_j}+\frac{k}{p_s})}
\big(
(u-1)^2-\beta(\widehat X_{\gamma_f})
+\sum_{i=1}^3 \big((u-1)-\beta(\widehat X_{\gamma_i})\big)
+3
\big)$$
$$=(u-1)u^{-(\frac{k}{p_i}+\frac{k}{p_j}+\frac{k}{p_s})}(u^2+u+1-\beta(\widehat X_{\gamma_f})-\sum_{i=1}^3 \beta(\widehat X_{\gamma_i})),$$
so that $\xi_k=u^2+u+1-\beta(\widehat X_{\gamma_f}\cup \widehat X_{\gamma_1}\cup \widehat X_{\gamma_2}\cup \widehat X_{\gamma_3})$ as expected.

Now we focus on $\eta_l$, with $1\leq l<k$. First of all, the contribution of the terms of the form $Q_k(\gamma)$ to $\beta(\mathcal A_k(f))$ is given by
$$
(u-1)^2\beta(\widehat X_{\gamma_f})Q_k(\gamma_f)+(u-1)^3\sum_{i=1}^3\beta(\widehat X_{\gamma_i})Q_k(\gamma_i),
$$
which is equal to
$$
(u-1)^2\beta(\widehat X_{\gamma_f})\sum_{1\leq l<k}u^{l-k}P_l(\gamma_f)+(u-1)^3\sum_{i=1}^3\beta(\widehat X_{\gamma_i})\sum_{1\leq l<k}u^{l-k}P_l(\gamma_i)
$$
by Lemma \ref{lemQ}.
In particular, using Lemma \ref{lemP},
\begin{itemize}
\item if $p_i\nmid l$ for $i=1,2,3$, or if $p_i\mid l$ but $p_{ij}\nmid l$ for $j\ne i$, then $P_l(\gamma_f)=P_l(\gamma_i)=0$ for $1\leq l<k$ and $i=1,2,3$, therefore $\eta_l=0$,
\item if $p_{ij}\mid l$ with $i,j\in \{1,2,3\}$ but $d \nmid l$, then $\eta_l=\beta(\widehat X_{\gamma_s})$, where $\{i,j,s\}=\{1,2,3\}$, 
\item if $d\mid l$, then $\eta_l=\beta(\widehat X_{\gamma_f})+\sum_{i=1}^3\beta(\widehat X_{\gamma_i})$.
\end{itemize}

\end{proof}

\begin{cor} 
Let $k\in\NN^*$. 
The degree of the virtual Poincar\'e polynomial of $\mathcal A_k(f)$ 
\coGR{satisfies the bound}
$$
\deg \beta(\mathcal A_k(f))\le
3
-\lfloor\frac{k}{p_1}\rfloor
-\lfloor\frac{k}{p_2}\rfloor
-\lfloor\frac{k}{p_3}\rfloor.
$$
\end{cor}

\begin{rem} Note in particular that
\begin{itemize}
\item  
if $p_1<p_2\le p_3$, then $\beta(\mathcal A_{p_1}(f))=\frac{u-1}{u}$.
\item
if $p_1=p_2<p_3$, then 
$\beta(\mathcal A_{p_1}(f))=\frac{(u-1)(1+u-\beta(\widehat{X}_{\gamma_3}))}{u^2}$. 
\item 
if $p_1=p_2=p_3$, then 
$\beta(\mathcal A_{p_1}(f))
=\frac{(u-1)(1+u+u^2-\beta(\widehat{X}_{\gamma_f}\cup \widehat X_{\gamma_1}\cup \widehat X_{\gamma_2}\cup \widehat X_{\gamma_3}))}{u^3}$. 
\end{itemize}
\end{rem}

However, in order to recover the weights of $f$, it will be enough to
concentrate the study on some specific parts of the zeta function. 
Actually, it is enough to recover the integers $p_1,p_2$ and $p_3$, 
together with $m_f(v)$, from
the zeta function of $f$ since $w_i=m_f(v)/p_i$ for $i=1,2,3$. 
Note that we already recover the multiplicity of $f$, 
that is $p_1$, as the order of
the zeta function. Moreover we know from Proposition \ref{signh} how
to recover the sign of $h(v)$.

In the particular case of $h(v)<0$ the function
$f$ has only \textit{simple} singularities in the sense of Arnold \cite{Arnold}
and we already know how to recover the
weights from \cite{simple}.

In the general situation, namely when $h(v)\geq 0$, we obtain 
$$e_f=\frac{h(v)}{m_f(v)}=1-\frac{1}{p_1}-\frac{1}{p_2}-\frac{1}{p_3}$$ 
by Proposition \ref{prop-posi} and $m_f(v)$ by Corollary \ref{corA}. 
\coGR{T}herefore \coGR{it is} sufficient to find $p_2$ or $p_3$ 
in order to recover all the weights. 
We are going to recover $p_2$ by a direct analysis of the first terms of 
the zeta function of $f$. 
More precisely, the idea is to recover $p_2$ in the zeta function as the first 
contribution that does not come from the vertex $(p_1,0,0)$ of the Newton 
polyhedron of $f$. 
Recall that $A(f)$ denote\coGR{s} the set of Fukui invariants 
$A(f)=\{k\in \NN^*:\mathcal
A_k\neq \emptyset \}$.

We treat the cases $p_1$ even and $p_1$ odd separately.
In \coGR{the} case $p_1$ is odd, note that $X_{\gamma_3}$ is not empty, 
so that $\beta(X_{\gamma_3})\neq 0$ and
therefore $A(f)\cap \mathbb N_{\geq p_{12}}=
\mathbb N_{\geq p_{12}} $. 

Set 
\begin{align*}
\alpha=&\min \{l\in \mathbb N^*:A(f)\cap \mathbb N_{\geq l}=
\mathbb N_{\geq l}\},\\
\beta=&\min \{l\in \mathbb N^*: \beta(\mathcal A_l)
\neq 0,\ p_1 \nmid l\}.
\end{align*}
As $p_1$ is odd, note that $\alpha \leq p_{12}$.
\begin{lem}
Assume that $p_1$ is odd. 
\begin{itemize}
\item If $p_1 \nmid \beta -1$, then $p_2=\beta$.
\item If $p_1 \mid \beta -1$ and $\beta -1< \alpha$, then
$p_2=\beta$.
\item  If $p_1 \mid \beta -1$ and $\beta -1= \alpha$, then either
$p_2=\beta -1$ or $p_2=p_3=\beta$.
\end{itemize}
\end{lem}

\begin{proof} 
If $p_1$ divides $p_2$, then 
$\alpha=p_2$ and $\beta=\alpha+1$, so that $p_1$ divides $\beta-1$. 
As $\beta=p_2$ if $p_1$ does not divide $p_2$, we obtain the first \coGR{claim}.

Now, if $p_1$ divides $\beta-1$, then either $p_1$ divides $p_2=\beta-1$ and
$\alpha=p_2$ or  $p_1$ divides $p_2-1$ and $p_2=\beta$.
\end{proof}

In particular, if $p_1 \mid \beta -1$ and $\beta = \alpha +1$, we obtain
two possibilities for the value of $p_2$. We show below that only
one of these possibilities gives the correct value for the sum
$\frac{1}{p_1}+\frac{1}{p_2}+\frac{1}{p_3}$, except in one particular
case that we need to treat separately.

\begin{lem}
Assume that $p_1 \mid \beta -1$ and $\beta = \alpha +1$. 
Assume moreover
that the value of $\frac{1}{p_1}+\frac{1}{p_2}+\frac{1}{p_3}$ is
given. Then we can decide whether
$p_2=\alpha$ or $p_2=\beta$, except in the cases $(p_1,p_2,p_3)=(3,4,4)$ or $(3,3,6)$.
\end{lem}

\begin{proof}
Assume that $(p_2,p_3)=(\beta,\beta)$
and $(p'_2,p'_3)=(\alpha,l)$ with $l\geq \alpha$ satisfying
$$
\frac{1}{p_2}+\frac{1}{p_3}=\frac{1}{p'_2}+\frac{1}{p'_3}.
$$
Then $l=\frac{\alpha(\alpha+1)}{\alpha-1}$ should be \coGR{an} integer\coGR{, 
and} therefore $\alpha=3$ and $l=6$. 
\end{proof}

In \coGR{the} case $p_1$ is even, it may 
\coGR{be possible} that $\alpha=\infty$ (only if $\beta(X_{\gamma_f})=0$), 
and even $\beta=\infty$ (only if $p_1$ divides $p_2$ and $p_3$). 
Therefore we need to take care also about the coefficients of
the zeta function, as described in Proposition \ref{prop-coeff}. 
Denote by $J$ a third root of unity different from 1 and set
$$
\delta=\min \{l\in \NN^*: -1 \textrm{ or } J \textrm{ is a zero of } 
\beta(\mathcal A_l) \}.
$$ 
\begin{lem}
Assume that $p_1$ is even. Then $p_2=\beta$ if $\alpha = \beta -1$ and 
$\beta(\mathcal A_{\alpha})=u ^{-\frac{\alpha}{p_1}}(u-1)$. 
In the other cases\coGR{,} $p_2=\min \{\alpha,\beta,\delta\}$.
\end{lem}

\begin{proof}


We \coGR{distinguish} cases according to the divisibility of
$p_2$ by $p_1$, and according to \coGR{equality} between $p_2$ and $p_3$.
\begin{itemize}
\item 
If $p_1$ does not divide $p_2$ and $p_2<p_3$, then $p_2=\beta$. 
Moreover $\beta \leq \delta$ by Proposition \ref{prop-coeff}. 
We claim that necessarily $\beta\leq\alpha$. 
\coGR{I}f not, the only possible case is $\alpha=\beta-1$ and so $p_1$
divides $\beta-1=p_2-1$. 
Since $p_2+1$ should be an element of $A(f)\cap \mathbb N_{\geq \alpha}$, 
then $p_1$ divides $p_2+1$ also, and so $p_1=2$. 
\coRR{Similarly $p_2+2$ should be an element of $A(f)\cap
\mathbb N_{\geq \alpha}$, and so $p_3=p_2+2$ and thus $(p_1,p_2,p_3)$ has the form $(2,p_2,p_2+2)$ with $p_2$ odd. Finally $p_2+4$ should be an element of $A(f)\cap \mathbb N_{\geq \alpha}$, so $p_2+4 \geq \lcm(2, p_2)=2p_2$. }
As a consequence $(p_1,p_2,p_3)=(2,3,5)$
and thus $h(v)<0$, which has been excluded in
our discussion.


\item If $p_1$ does not divide $p_2$ and $p_2=p_3$, then $p_2=\beta$
  and $\beta 
  \leq \delta$ by Proposition \ref{prop-coeff}. Moreover $\beta \leq \alpha$ unless $p_1$ divides $p_2-1$ and
$\beta(X_{\gamma_1})\neq 0$. In that \coGR{case,} $\beta=\alpha+1\leq
\delta$ and moreover $\beta(\mathcal
A_{\alpha})=u ^{-\frac{\alpha}{p_1}}(u-1)$ by Proposition
\ref{prop-coeff}. This case gives the
exceptional case in the statement of the lemma.

\item If $p_1$ divides $p_2$, assume first that $p_2<p
_3$. Then $p_2=\delta$ if $\beta(X_{\gamma_3})= 0$, and then $\delta \leq
\min \{\alpha, \beta\}$. If however $\beta(X_{\gamma_3})\neq 0$, then $p_2=\alpha \leq \min\{\beta, \delta\}$.

In the latter case, note that $\beta(\mathcal
A_{\alpha})=u ^{-\frac{\alpha}{p_1}-1}(u^2-1-(u-1)\beta(\widehat
X_{\gamma_3}))$ with $\beta(\widehat X_{\gamma_3})$ even (indeed $\beta(\widehat
X_{\gamma_3})$ is equal to the number of real solutions of a real
polynomial of even degree not vanishing at zero, with only simple real
roots because $f$ has isolated singularities). 

\item If $p_1$ divides $p_2=p_3$, we have to distinguish the cases
  where $\alpha$ is infinite or not. If $\alpha \neq \infty$ then $p_2=\alpha \leq \{\beta,
\delta\}$ since $[X_{\gamma_f}]\neq 0$.
Finally if $p_1$ divides $p_2=p_3$ and $\alpha = \infty$, then
$p_2<\beta$ and $J$ is a zero of $\beta(\mathcal A_{p_2})$ 
by Proposition
\ref{prop-coeff}, therefore $p_2=\delta$.
\end{itemize}
\end{proof}

\begin{thm}\label{w3}  Convenient weighted
  homogeneous polynomials which share the same zeta functions have the
  same weights.
\end{thm}

\begin{proof} 
If $h(v)<0$ we refer to \cite{simple}. 
Otherwise, the discussion in section \ref{reco} 
show\coGR{s} how to recover $p_1$, $p_2$ and
$\frac{1}{p_1}+\frac{1}{p_2}+\frac{1}{p_3}$ 
except in the particular case where $p_1=3$ and
$\frac{1}{p_1}+\frac{1}{p_2}+\frac{1}{p_3}=\frac{5}{6}$. 
Therefore it suffices to distinguish the cases 
$(p_1,p_2,p_3)=(3,4,4)$ and $(3,3,6)$.

A direct computation shows that $\beta(\mathcal A_{4})=u ^{-3}(u-1)^2$
if $(p_1,p_2,p_3)=(3,3,6)$ whereas 
$\beta(\mathcal A_{4})=u ^{-3}(u^2-1-(u-1)\beta(\widehat X_{\gamma_1}))$
if $(p_1,p_2,p_3)=(3,4,4)$, 
so that the spaces of arcs of level $4$ are different, 
except when $\beta(\widehat X_{\gamma_1})=2$. 
However $\beta(\mathcal A_{5})=u ^{-4}(u-1)^2$ 
if $(p_1,p_2,p_3)=(3,3,6)$ 
whereas $\beta(\mathcal A_{5})=u^{-4}(u-1)\beta(\widehat X_{\gamma_1})$ 
if $(p_1,p_2,p_3)=(3,4,4)$,
so at \coGR{level} $5$ the spaces of arcs are different\coGR{.}
\end{proof}


\end{document}